\documentclass[11pt,reqno]{amsart}
\usepackage{a4wide}

\numberwithin{equation}{section}
\usepackage{mathrsfs}
\usepackage{amsfonts}
\usepackage{amsmath}
\usepackage{stmaryrd}
\usepackage{amssymb}
\usepackage{amsthm}
\usepackage{mathrsfs}
\usepackage{url}
\usepackage{amsfonts}
\usepackage{amscd}
\usepackage{indentfirst}
\usepackage{enumerate}
\usepackage{amsmath,amsfonts,amssymb,amsthm}
\usepackage{amsmath,amssymb,amsthm,amscd}
\usepackage{graphicx,mathrsfs}
\usepackage{appendix}

\usepackage[numbers,sort&compress]{natbib}

\newcommand{\R}{\mathbb{R}}

\renewcommand{\theequation}{\arabic{section}.\arabic{equation}}

\newtheorem{theo}{Theorem}[section]
\newtheorem{lem}[theo]{Lemma}

\newtheorem{defi}[theo]{Definition}
\newtheorem{rema}[theo]{Remark}

\begin{document}
\title[{multi-peak positive solutions for the nonlinear Kirchhoff equations}]
{The existence of multi-peak positive solutions for nonlinear Kirchhoff equations on $\R^3$
}

\author[H. Chen]{Hong Chen}

\address[Hong Chen]{School of Mathematics and Statistics, Central China Normal University, Wuhan 430079, China}
\email{chenhongcwh@163.com}

\author[Q. Hua]{Qiaoqiao Hua$^{\dag}$}
\address[Qiaoqiao Hua]{School of Mathematics and Statistics, Central China Normal University, Wuhan 430079, China}
\email{hqq@mails.ccnu.edu.cn}

\date{\today}


\begin{abstract}
In this work,  we study the following Kirchhoff equation
\begin{equation*}
\begin{cases}
-\left(\varepsilon^2 a+\varepsilon b\int_{\mathbb R^3}|\nabla u|^2\right)\Delta u +u =Q(x)u^{q-1},\quad u>0,\quad x\in {{\R}^{3}},\\
u\to 0,\quad \text{as}\ |x|\to +\infty,
\end{cases}
\end{equation*}
where $a,b>0$ are constants, $2<q<6$, and $\varepsilon>0$ is a parameter. Under some suitable assumptions on the function $Q(x)$, we obtain that the equation above has positive multi-peak solutions concentrating at a critical point of\ $Q(x)$ for $\varepsilon>0$ sufficiently small, by using the Lyapunov-Schmidt reduction method. We extend the result in (Discrete Contin. Dynam. Systems 6(2000), 39--50) to the nonlinear Kirchhoff equation.
\end{abstract}

\keywords {\textbf{Kirchhoff equations, Multi-peak positive solutions, Lyapunov-Schmidt reduction method}}

\thanks{$^{\dag}$ Corresponding author: Qiaoqiao Hua}
\maketitle

\section{Introduction and main results}
\setcounter{equation}{0}
In this paper, we consider the following Kirchhoff equation
\begin{equation}\label{11}
\begin{cases}
-\left(\varepsilon^2 a+\varepsilon b\int_{\mathbb R^3}|\nabla u|^2\right)\Delta u +u =Q(x)u^{q-1},\quad u>0,\quad x\in {{\R}^{3}},\\
u\to 0,\quad \text{as}\ |x|\to +\infty,
\end{cases}
\end{equation}
where $a,b>0$ are constants, $2<q<6$, $\varepsilon>0$ is a parameter and $Q(x):{\R}^{3}\rightarrow {\R}$ is a smooth bounded function.

Problem\ (\ref{11}) and its variants have been studied extensively in the literature. To extend the classical D'Alembert's wave equations for free vibration of elastic strings, Kirchhoff\ \cite{20} first proposed the following time-dependent wave equation
\begin{equation}\label{12}
  \rho \frac{\partial^2 u}{\partial^2 t}-\left( \frac{P_0}{h}+\frac{E}{2L}\int_0^L \left|\frac{\partial u}{\partial x}\right|^2\,\mathrm{d}x\right)\frac{\partial^2 u}{\partial^2 x}=0.
\end{equation}
Bernstein\ \cite{4} and Pohozaev\ \cite{31} studied the above type of Kirchhoff equations quite early. In order to study the problem\ (\ref{12}) preferably,\ Lions\ \cite{26} introduced an abstract functional framework to this problem. After that, many interesting results of Kirchhoff equations can be found in e.g.\ \cite{2,11,12,15,16,43} and the reference therein. From a mathematical point of view, Kirchhoff equation is nonlocal, in the sense that the term\ $\left(\int_{\mathbb R^3} \left| \nabla u\right|^2\right)\Delta u$ depends not only on~$\Delta u$, but also on the integral of ~$\left|\nabla u\right|^2$~ over the whole space. This feature brings new mathematical difficulties, which makes the study of Kirchhoff type equations particularly meaningful. We refer to e.g.\ \cite{17,18,19,30} for mathematical researches on Kirchhoff type equations in bounded domains and in the whole space.

In fact, equation (\ref{11}) is closely related to the\ Schr\"odinger equations. When\ $a=1$ and $b=0$, equation\ (\ref{11}) is reduced to the perturbed Schr\"odinger equation. 
Kwong\ \cite{21} considered the classical Schr\"odinger equation
\begin{equation}\label{14}
  -\Delta u+u=u^p,\quad x\in \mathbb R^N,
\end{equation}
where $1<p<+\infty$ if\ $N=1,2$, and  $1<p<\frac{N+2}{N-2}$ if\ $N\ge 3$. Equation\ (\ref{14}) has a unique radial symmetric and nondegenerate positive solution. Based on this property, Cao, Noussair and Yan\ \cite{6} proved the existence of multi-peak solutions for equation
\begin{equation}\label{15}
-\Delta u+\lambda^2 u=Q(x)|u|^{q-2}u,\quad x\in\ \mathbb R^N,
 \end{equation}
where $\lambda \not=0,\ N\ge3$ and $2<q<2N/(N-2)$.

Dancer and\ Yan\ \cite{24} studied the following equation
\begin{equation}\label{16}
\begin{cases}
-\varepsilon^2 \Delta u +u =Q(y)u^{p-1},\quad u>0,\quad y\in {{\R}^{N}},\\
u\to 0,\quad \text{as}\ |y|\to +\infty,
\end{cases}
\end{equation}
where\ $\varepsilon>0$ is a parameter, $2<p<+\infty$ if $N=2$ and $2<p<2N/(N-2)$ if $N>2$. They not only proved that the~Schr\"{o}dinger~equation has positive multi-peak solutions concentrating at a designated saddle point or a strictly local minimum point of $Q(y)$ in $\mathbb{R}^N$, but also showed that there is no multi-peak solution concentrating at a strictly local maximum point of $Q(y)$ in $\mathbb{R}^N$. Besides, many interesting results of multi-peak solutions can be found in e.g.\ \cite{38,37,7,14,10,13,35,27,29} and the reference therein.

Based on the uniqueness and nondegeneracy property of equation (\ref{14}), the authors in \cite{23} proved the uniqueness and nondegeneracy of positive solutions for equation
\begin{equation}\label{17}
  - \left( a +  b\int_{\mathbb{R}^3} \left| \nabla u \right|^2 \right)\Delta u + u = u^p,\quad u>0\quad \text{in\ $\mathbb R^3$ },
\end{equation}
where\ $1<p<5$. Then, by using Lyapunov-Schmidt reduction method, they constructed the existence and the uniqueness of single peak solutions to equation
\begin{equation}\label{18}
  - \left(\varepsilon^2 a + \varepsilon b\int_{\mathbb{R}^3} \left| \nabla u \right|^2 \right)\Delta u + V(x)u = u^p,\quad u>0\quad \text{ in\ $\mathbb R^3$ },
\end{equation}
wher\ $1<p<5$ and $\varepsilon>0$ is sufficiently small.

Luo, Peng and Wang\ \cite{25} proved equation\ (\ref{18}) has positive multi-peak solutions concentrating at different points if $\varepsilon>0$ is sufficiently small. It should be pointed that, they constructed the multi-peak solutions of equation\ (\ref{18}) based on the following systems
\begin{equation}\label{19}
-\left(a+b\sum_{j=1}^{k}{\int_{{\R}^{3}}{{\left| \nabla u_{j} \right|}^{2}}}\right){\Delta u_{i}}+V{\left({a_i}\right)}{u_{i}}={(u_{i})^{p}},\ u_{i}>0,\ i=1,\ldots,k,\ x\in {{\R}^{3}}.\\
\end{equation}
They also showed that\ $a_i\ (1\le i\le k)$ are critical points of\ $V$, and there exist the multi-peak solutions of the form $u_{\varepsilon}=\sum_{i=1}^k u_i\bigg( \frac {x-y_{\varepsilon,i}} {\varepsilon} \bigg)+\varphi_\varepsilon$.

Note that in\ \cite{25}, they didn't consider the existence of multi-peak solutions concentrating at the same point ($a_1=\cdots=a_k$) to equation\ (\ref{18}). When\ $a_1=a_2=\cdots=a_k$, $|y_{\varepsilon,i}-y_{\varepsilon,j} |\to 0\ (i\not=j)$ as $\varepsilon\to 0$, but we cannot get the result of $|y_{\varepsilon,i}-y_{\varepsilon,j} |/\varepsilon \to +\infty\ (i\not=j)$ as $\varepsilon\to 0$. Therefore, we must impose additional conditions on\ $y_{\varepsilon,i}$.

Recently, Cui et al. \cite{32} proved the existence and local uniqueness of normalized multi-peak solutions to the following Kirchhoff equation
\begin{equation}
\begin{cases}
-\left(a+b_\lambda \int_{\mathbb R^3}|\nabla u_\lambda|^2\right)\Delta u_\lambda +\left(\lambda+V(x)\right)u_\lambda =\beta_\lambda u_\lambda^p,\ u_{\lambda}>0,\ x\in {{\R}^{3}},\\
u_{\lambda}\in{H^{1}{\left( {\R^3} \right)}},
\end{cases}
\end{equation}
where $a>0$,\ $1<p<5$,\ $\lambda,\ b_\lambda,\ \beta_\lambda>0$ are parameters, $\int_{\mathbb R^3} u_\lambda^2=1$, and $V(x):{\R}^{3}\rightarrow {\R}$ is  a bounded continuous function.

Inspired by the literatures \cite{24,23,25,32}, we apply the conditions of\ $Q(x)$ in\ \cite{24} to\ ${\R}^{3}$. Then, we consider the existence of positive multi-peak solutions to the equation\ (\ref{11}) concentrating at a critical point of $Q(x)$.

Now we give some definitions and assumptions.

\begin{defi}
Let\ $m\in N_+$,\ $a_0\in {\R}^{3}$, we say that\ $u_\varepsilon$ is a m-peak solution of\ (\ref{11}) if\ $u_\varepsilon$ satisfies

$(i)$  $u_\varepsilon$ has\ $m$ local maximum points\ $y_{\varepsilon,i}\in {\R}^{3},\ i=1,\ldots,m$, satisfying
\begin{equation*}
  y_{\varepsilon,i}\to a_0
\end{equation*}
as $\varepsilon\to 0$ for each $i$;

$(ii)$ For any given\ $\tau>0$, there exists\ $R\gg 1$, such that
\begin{equation*}
  |u_\varepsilon(x)|\le\tau,\quad x\in\mathbb R^3\setminus \cup_{i=1}^mB_{R\varepsilon}(y_{\varepsilon,i});
\end{equation*}

$(iii)$ There exists\ $C>0$ such that
\begin{equation*}
  \int_{\mathbb R^3}\left( \varepsilon^2a|\nabla u_\varepsilon|^2+u_\varepsilon^2\right)\le C\varepsilon^3.
\end{equation*}
\end{defi}

We assume that $Q(x)$ satisfies the following conditions:

$(Q_1)$ $Q(x)$ is a smooth bounded function in\ $\mathbb{R}^3$.

$(Q_2)$ $x_0=0 $ is a critical point of\ $Q(x)$.

$(Q_3)$ $Q(x)$ has the following expansion (after suitably rotating the coordinate system)
\begin{equation}\label{111}
  Q(x)=Q(0)+P_1(x')-P_2(x'')+R(x),\quad x\in B_{\delta}(0),
\end{equation}
where\ $Q(0)>0 $, $\delta>0 $, $x=(x',x'') $, $x'=(x_1,...,x_t) $, $x''=(x_{t+1},...,x_3) $, $t\in \{1,2,3\} $,\ $P_1$ and $P_2$ satisfy
\begin{equation}\label{112}
  P_1(x')=\lambda|x'|^{h_1},\quad |x'|\le \delta,
\end{equation}
\begin{equation}\label{114}
  \langle DP_2(x''),x'' \rangle \ge \lambda|x''|^{h},\quad |x''|\le \delta,
\end{equation}
\begin{equation}\label{115}
  |D^mP_2(x'')|=O(|x''|^{h-m}), \quad m=0,...,[h], \quad |x''|\le \delta,
\end{equation}
for some\ $h_1\ge h\ge 2$ and some positive constant $\lambda>0$, and
\begin{equation}\label{116}
R(x)=O(|x|^{h_1+\sigma}),\ \text{for\ some}\ \sigma>0\ as\ |x|\to 0.
\end{equation}
\begin{rema}
$Q(x)=\sin|x|^2+1,\ x\in \mathbb R^3,$ satisfies the conditions\ $(Q_1)$--$(Q_3)$.
\end{rema}

Assume\ $u,v\in H^1(\mathbb R^3)$, denote
\begin{gather*}
u_{\varepsilon,y}(x)=u\bigg(\frac{x-y}{\varepsilon}\bigg),\ y\in \mathbb R^3,\\
  \langle u,v \rangle_\varepsilon=\int_{\mathbb R^3} \left(\varepsilon^2a\nabla u\nabla v+uv\right),\quad \|u\|_{\varepsilon}=  \langle u,u\rangle_\varepsilon^{ 1/2},\\
  H_\varepsilon=\{u\in H^1(\mathbb R^3):\|u\|_{\varepsilon}<+\infty  \}.
\end{gather*}
The energy functional corresponding to equation\ (\ref{11}) is
\begin{equation}\label{117}
  I_\varepsilon(u)=\frac 12 \|u\|_{\varepsilon}^2+\frac{\varepsilon b}{4}\bigg( \int_{\mathbb R^3}|\nabla u|^2 \bigg)^2-\frac 1q \int_{\mathbb R^3} Q(x)u^q_+,\quad u\in H_\varepsilon,
\end{equation}
where\ $u_+=\max(u,0)$. It is standard to verify that $I_\varepsilon\in C^2( H_\varepsilon)$. So we just need to find a critical point of $I_\varepsilon\in C^2( H_\varepsilon)$.

For\ $k\in N_+$, let $w$ be the unique positive radial solution (see Lemma\ \ref{lm2.2}) to equation
\begin{equation}\label{118}
 - \left( {a + bk\int_{{\mathbb{R}^3}} {{{\left| {\nabla u} \right|}^{^2}}} } \right)\Delta u + u = Q\left( {0} \right) u ^{q - 1}.
\end{equation}
Then we want to construct k-peak solutions to equation\ (\ref{11}) concentrating at the critical point $x_0=0$ of $Q(x)$ by using the uniqueness and nondegeneracy property of $w$.

Our main results are as follows.
\begin{theo}\label{thm1.1}
Assume that $Q(x)$ satisfies\ $(Q_1)$--$(Q_3)$. Then, for any\ $k\in N_+$, there exists\ $\varepsilon_0=\varepsilon(k)$ such that for\ $\varepsilon\in(0,\varepsilon_0]$, equation\ (\ref{11}) has at least one k-peak solution of the form
\[u_\varepsilon=\sum_{i=1}^k w_{\varepsilon,y_{\varepsilon,i}}+\varphi_\varepsilon,   \]
where\ $y_{\varepsilon,i}\in\mathbb{R}^3,\ i=1,\ldots,k,\ \varphi_\varepsilon\in H^1(\mathbb{R}^3)$, and as\ $\varepsilon\to 0$,
\begin{gather*}
 y_{\varepsilon,i}\to 0,\ i=1,\ldots,k, \\
  \frac{|y_{\varepsilon,i}-y_{\varepsilon,j}|}{\varepsilon}\to +\infty,\ i\not= j,\ i,j=1,\ldots,k,\\
  \left\|\varphi_\varepsilon \right \|_\varepsilon=o(\varepsilon^{\frac 32}).
\end{gather*}
\end{theo}
\begin{rema}
Theorem\ \ref{thm1.1} extend the result got by Dancer and Yan\ \cite{24} about the existence of solutions for the nonlinear Schr\"odinger equation to the nonlinear Kirchhoff equation\ (\ref{11}).
\end{rema}

In this paper, we prove Theorem\ \ref{thm1.1} by using Lyapunov-Schmidt reduction method. Since there is a nonlocal term, we encounter some new difficulties which involve some complicated and technical estimates. To our knowledge, the result we obtain is new.



Our notations are standard. We use $B_r(x)$ (and $\overline{ B_r(x)}$) to denote open (and close) balls in $\mathbb R^3$ centred at $x$ with radius $r$, and $B_r^c(x)$ to denote the complementary set of\ $B_r(x)$ in\ $\mathbb R^3$. Unless otherwise stated, we write $\int u$  to denote Lebesgue integrals over\ $\mathbb R^3$, and $\|u\|_{L^p}$, $\|u\|_{H^1}$ to mean\ $L^p$-norm, $H^1$-norm respectively. We will use $C,C_j(j\in N)$ to denote various positive constants, and $O(t)$, $o(t)$, $o_t(1)$,\ $o_R(1)$  to mean $|O(t)|\le C|t|$, $o(t)/t\to 0$, $o_t(1)\to 0$ as $t\to 0$ and $o_R(1)\to 0$ as\ $R\to +\infty$ respectively.

This paper is organized as follows. In section 2, we recall some definitions and lemmas. In section 3, we give the finite dimensional reduction process and in section 4, we prove Theorem\ \ref{thm1.1}.


\section{ Some preliminaries}
In this section, we introduce some preliminaries.
\begin{lem}\label{lm2.1}(\ \cite{25}, Lemma\ 2.1)
 For any\ $2\le q\le 6$, there exists a constant\ $C>0$ independent of\ $\varepsilon$, such that
 \begin{equation}\label{21}
  {\left\| u \right\|_{{L^q}}} \le C{\varepsilon ^{\frac{3}{q} - \frac{3}{2}}}{\left\| u \right\|_\varepsilon },\ \forall \ u\in H_\varepsilon.
 \end{equation}
\end{lem}

Before stating the lemma that follows, we first give a truth that $U$ is a unique positive radial solution to equation
\begin{align}\label{24}\left\{
                \begin{array}{ll}
                    -\Delta u+u = u^{q-1},\ x\in\mathbb R^3,\\
                  u  \in   H^1(\mathbb {R}^3), \\
                   u(0)=\underset{x\in\mathbb R^3}{\max}\ u(x),
                \end{array}
              \right.
\end{align}
which satisfies
\begin{align*}\left\{
                \begin{array}{ll}
                  \mathop{\lim }\limits_{|x| \to \infty }|x|e^{|x|}U(|x|)=C>0, &  \\
                  \mathop{\lim }\limits_{|x| \to \infty }\frac{U'(|x|)}{U(|x|)}=-1. &
                \end{array}
              \right.
\end{align*}

\begin{lem}\label{lm2.2}(\ \cite{23}, Theorem\ 1.2)
For any\ $k\in N_+$, $a,b>0,\ 2<q<6$, and $Q(x)$ satisfying\ $(Q_1)$--$(Q_3)$, there exists a unique positive radial solution\ $w\in H^1(\mathbb R^3)$ satisfying
\begin{equation}\label{25}
  -\left(a+bk\int_{\mathbb{R}^3}|\nabla u |^2\right)\Delta u+u=Q(0)u^{q-1}.
\end{equation}
Moreover,\ $w$ is nondegenerate in $H^1(\mathbb R^3)$ in the sense that there holds
\begin{equation*}
 \mathrm{Ker}\, \mathcal L= \mathrm{span} \Big\{\frac{\partial w}{\partial x_1},\frac{\partial w}{\partial x_2},\frac{\partial w}{\partial x_3}\Big\},
\end{equation*}
where\ $ \mathcal L:H^1(\mathbb R^3)\to H^1(\mathbb R^3)$ is the linear operator defined as
\begin{equation}
    \mathcal L\varphi=-\left(a+bk\int|\nabla w|^2\right)\Delta \varphi-2bk\bigg(\int\nabla w\nabla\varphi\bigg)\Delta w+\varphi-(q-1)Q(0)w^{q-2}\varphi, \forall \varphi\in H^1(\mathbb R^3).
  \end{equation}
\end{lem}
\noindent
\begin{proof} Assume\ $u$ is an arbitrary positive solution to equation (\ref{25}). Write
\begin{equation}\label{27}
  c_k=a+bk\int|\nabla u|^2.
\end{equation}
Let\ $\tilde {u}(x)=(Q(0))^{\frac{1}{q-2}}u(\sqrt{c_k}x)$, then\ $\tilde{u}(x)$ satisfies
\begin{equation*}
  -\Delta{\tilde {u}}(x)+{\tilde {u}}(x)={\tilde {u}}^{q-1}(x),
\end{equation*}
in the sense that\ ${\tilde {u}}(x)$ solves equation\ (\ref{24}). Thus, by the uniqueness of $U$ we have
\begin{equation*}
 {\tilde {u}}(x)=U(x-z),
\end{equation*}
for some $z\in \mathbb R^3$. The relationship between $u$ and\ $\tilde u$ implies that
\begin{equation}\label{28}
 u(x)=(Q(0))^{-\frac{1}{q-2}}U\bigg(\frac{1}{\sqrt{c_k}}x-z\bigg)=:\lambda_1 U\bigg(\frac{1}{\sqrt{c_k}}x-z\bigg).
\end{equation}
By\ (\ref{28}), we obtain
\begin{equation}\label{29}
 \int|\nabla u|^2=\lambda_1^2\sqrt{c_k}\int|\nabla U|^2.
\end{equation}
Combining\ (\ref{27}) and (\ref{29}) yields
\begin{equation}\label{210}
   c_{k}=a+bk\lambda_1^2\|\nabla U\|^2_{L^2}\sqrt{c_k}=:a+\tilde{b}\sqrt{c_k}.
\end{equation}
Since $c_{k}>0$, equation (\ref{210}) is uniquely solved by
\begin{equation}\label{211}
  \sqrt{c_k}=\frac {\tilde{b}+\sqrt {\tilde{b}^2+4a}}{2},
\end{equation}
which shows that $c_k$ is independent of the choice of positive solutions to equation\ (\ref{25}). Additionally, since equation\ (\ref{24}) has a unique positive radial solution,\ (\ref{28}) implies that equation\ (\ref{25}) has a unique positive radial solution\ $w$ when\ $z=0$.
\end{proof}

The above proof implies that\ $w(x)=\lambda_1 U(\eta x)$ where\ $\eta =\frac{1}{\sqrt{c_k}}$. And since\ $U$ decays exponentially at infinity, we infer that
\begin{equation}\label{213}
 \nabla w(x),\ w(x)=O( \ |x|^{-1}e^{-\eta |x|})\ \text{as}\ |x|\to \infty.
\end{equation}

\begin{lem}\label{lm2.3}(\ \cite{1}, Lemma\ 3.7)
Assume\ $u,u':\mathbb R^n\rightarrow \mathbb R$ are positive radial continuous functions satisfying
\[    u(x)\sim |x|^ae^{-b|x|},\ u'(x)\sim |x|^{a'}e^{-b'|x|} \   (|x|\to \infty),                          \]
where\ $a,a'\in \mathbb R,\ b,b'>0$. Let $\xi\in \mathbb{R}^n$ tend to infinity and $u_\xi(x)=u(x-\xi)$. Then the following asymptotic estimates hold:

$(i)$ If\ $b<b'$, then
\[ \int_{\mathbb R^n} u_\xi u'\sim e^{-b|\xi|} |\xi|^a.                            \]
If\ $b>b'$, then replace\ $a$ and $b$ with $a'$ and $b'$.

$(ii)$ If\ $b=b'$, suppose, for simplicity, that $a\ge a'$, then
\begin{equation*}
  \int_{\mathbb R^n}u_\xi u'\sim
\left\{
     \begin{array}{ll}
       e^{-b|\xi|}|\xi|^{a+a'+\frac {n+1}{2}}, & \hbox{$a'>-\frac{n+1}{2}$;} \\
       e^{-b|\xi|}|\xi|^a \log|\xi|, & \hbox{$a'=-\frac{n+1}{2}$;} \\
       e^{-b|\xi|}|\xi|^a , & \hbox{$a'<-\frac{n+1}{2}$.}
     \end{array}
   \right.
\end{equation*}
\end{lem}

Combining\ (\ref{213}) and Lemma\ \ref {lm2.3} yields, for any\ $r,s>0$ with $r\not= s$, we have, as\ $\varepsilon \to 0$,
\begin{equation}\label{214}
  \int_{\mathbb R^3}w_{\varepsilon,y_i}^r w_{\varepsilon,y_j}^s =O\Big( \varepsilon^3 e^{-\min\{r,s\}\frac{\eta|y_i-y_j|}{\varepsilon}}\Big|\frac{y_i-y_j}{\varepsilon}\Big|^{-\min\{r,s\}}\Big) ,
\end{equation}

\begin{equation}\label{215}
  \int_{\mathbb R^3}|\nabla w_{\varepsilon,y_i}  \nabla w_{\varepsilon,y_j} | =O\Big( \varepsilon e^{-\frac{\eta|y_i-y_j|}{\varepsilon}}\Big).
\end{equation}
Particularly, there exists some constants\ $C>0$, such that
\begin{equation}\label{216}
 \int_{\mathbb R^3}w_{\varepsilon,y_i}^rw_{\varepsilon,y_j}^s \le C  \varepsilon^3 e^{-\min\{r,s\}\frac{\eta|y_i-y_j|}{\varepsilon}},
\end{equation}
\begin{equation}\label{217}
  \int_{\mathbb R^3}|\nabla w_{\varepsilon,y_i}  \nabla w_{\varepsilon,y_j}|\le C \varepsilon e^{-\frac{\eta|y_i-y_j|}{\varepsilon}} .
\end{equation}

In the following sections, we will use inequality\ (\ref{21}), (\ref{216}) and (\ref{217}) repeatedly. When\ $r=s$, the situation is complicated. But when\ $r=s\ge1$, inequality (\ref{216}) still holds.

\begin{defi}\label{de2.4}(\ \cite{22}, Definition\ B.1)
Let $Y$ and $A$ be closed subsets of a topological space $X$. Then $Cat_X(A,Y)$ is the least integer $k$ such that\ $A=\mathop{\cup_{j=0}^k}A_j$, where, for $0\le j\le k$, $A_j$ is closed, and there exists\ $h_j\in C\big([0,1]\times A_j,X\big)$ such that

$(i)$  $ h_j(0,x)=x$ for $x\in A_j,\ 0\le j\le k$;

$(ii)$ $ h_0(1,x)\in Y$ for $x\in A_0$ and $h_0(t,x)=x$ for $x\in A_0\mathop{\cap}Y$ and $t\in [0,1]$;

$(iii)$ $ h_j(1,x)=x_j$ for $x\in A_j$ and some $ x_j\in X,\ 1\le j\le k$.\\
Particularly, if $Y$ is empty, then write\ $Cat_X(A)=Cat_X(A,\varnothing)$.
\end{defi}
From the Definition \ref{de2.4}, we see $Cat_X(A,Y)\ge 1$  if $A$ can not be deformed into a subset of $Y$ within $X$.

\begin{lem}\label{lm2.5}(\ \cite{22}, Proposition\ 2.2)
Suppose that $F(x)$ is a $C^2$ function defined in a bounded domain\ $\Omega\subset \mathbb R^{3k}$. If $F$ satisfies either $F(x)>c$ or $\frac{\partial F(x)}{\partial n}>0$ at each $x\in \partial \Omega$, where $n$ is the outward unit normal of $\partial \Omega$ at $x$, then
 \begin{equation}
  \#\big\{x: DF(x)=0,\ x\in F^c\big\}\ge Cat_{F^c}(F^c),
 \end{equation}
where\ $F^c=\big\{x: x\in \Omega,\ F(x)\le c\big\}$. In particular, $F(x)$ has at least one critical point in\ $F^c$.
\end{lem}
\noindent

\begin{lem}\label{lm2.6}(\ \cite{22}, Proposition\ 2.3)
Suppose that $F(x)$ is a $C^2$ function defined in a bounded domain\ $\Omega\subset \mathbb R^{3k}$. Let\ $c_1,\ c_2$ be two constants such that neither $c_2$ nor $c_1$ is a critical value of $F(x)$. If $F$ satisfies either\ $F(x)<c_1$ or\ $\frac{\partial F(x)}{\partial n}>0$ for each $x\in \partial \Omega$, then
 \begin{equation}
  \#\big\{x: DF(x)=0,\ x\in F^{c_2}\backslash F^{c_1}\big\}\ge Cat_{F^{c_2}}(F^{c_2},F^{c_1}).
 \end{equation}
In particular, if $F^{c_2}$ cannot be deformed into $F^{c_1}$, $F$ has at least one critical point in $F^{c_2}\backslash F^{c_1}$.
\end{lem}
\noindent

\begin{lem}\label{lm4.2}(\ \cite{8}, Proposition\ B.1)
Suppose that $X$ and $\Gamma$ are two compact sets in\ $\mathbb R^{3}$ satisfying\ $\Gamma \subset X$. Let
\begin{equation}
K=\underbrace {X\times \cdots \times X}_k,
\end{equation}
\begin{equation}
L_1=\underbrace {\Gamma \times X\times \cdots \times X}_k \cup \underbrace {X \times \Gamma \times \cdots \times X}_k \cup \underbrace {X \times \cdots X \times \Gamma}_k,
\end{equation}
\begin{equation}
L_2=L_1 \cup D, \quad D=\big\{Y=(y_1,\ldots,y_k):\ y_i=y_j\ \text{for\ some}\ i\not =j\big\},
\end{equation}
If\ $H^m(X,\Gamma)\not =0$ for some\ $m\ge 1$, then\ $H_{*}(K,L_2)\not =0$. In particular, $K$ cannot be deformed into\ $L_2$.
\end{lem}

\section{Finite dimensional reduction}

In this section we complete Step 1 as mentioned in Section 1. Denote
\begin{gather*}
 Y=(y_1,\ldots,y_k)\in \mathbb {R}^3\times \cdots \times\mathbb {R}^3,\quad W_{\varepsilon ,Y}=\sum_{i=1}^k w_{\varepsilon,y_i},\\
  D^k_{\varepsilon,\delta}=\Big\{Y: y_i\in \overline{ B_{\delta}(0)},\ \frac{\eta|y_i-y_j|}{\varepsilon}\ge R_1 ,\   i,j=1,\ldots,k,\ i\not=j\Big\},
  \end{gather*}
where\ $R_1>0 $ is a fixed large constant. Let
\begin{equation*}
 E_{\varepsilon ,Y}^k = \bigg\{ {\varphi  \in {H_\varepsilon }:{{\bigg\langle {\varphi ,\frac{{\partial {w_{\varepsilon ,{y_i}}}}}{{\partial y_{ij}}}} \bigg\rangle }_\varepsilon } = 0,\ i = 1, \ldots, k,\ j = 1,2,3} \bigg\},
\end{equation*}
and define
\begin{equation*}
 {J_\varepsilon }\left( {Y,\varphi } \right) = {I_\varepsilon }\left( {{W_{\varepsilon ,Y}} + \varphi } \right),\ \forall \left( {Y,\varphi } \right) \in D_{\varepsilon ,\delta }^k \times E_{\varepsilon ,Y}^k.
\end{equation*}
Expand\ $J_{\varepsilon}(Y,\varphi)$ near $\varphi=0$ for each fixed $Y$:
\begin{equation}\label{120}
 {J_\varepsilon }\left( {Y,\varphi } \right) = {J_\varepsilon }\left( {Y,0} \right) + {l_{\varepsilon ,Y}}\left( \varphi  \right) + \frac{1}{2}Q_{\varepsilon,Y}(\varphi) + {R_{\varepsilon ,Y}}\left( \varphi  \right),
\end{equation}
where
\begin{equation}\label{31}
\begin{aligned}
{l_{\varepsilon,Y}}\left(\varphi\right)
&=\left\langle{{I_\varepsilon}'\left({{W_{\varepsilon ,Y}}}\right),\varphi} \right\rangle\\
&={\left\langle{{W_{\varepsilon ,Y}},\varphi}\right\rangle_\varepsilon}+\varepsilon b{\int{\left|{\nabla{W_{\varepsilon,Y}}}\right|}^2}\int{\nabla{W_{\varepsilon,Y}}\nabla\varphi}-\int{Q\left(x\right)} W_{\varepsilon,Y}^{q - 1}\varphi,\\
\end{aligned}
\end{equation}
\begin{equation}\label{32}
\begin{aligned}
Q_{\varepsilon,Y}(\varphi)
=&\left\langle {{I_\varepsilon }''\left( {{W_{\varepsilon ,Y}}} \right)\left[ \varphi  \right],\varphi } \right\rangle\\
=&{\left\langle {\varphi ,\varphi } \right\rangle _\varepsilon } + 2\varepsilon b{\bigg( {\int {\nabla {W_{\varepsilon ,Y}}\nabla \varphi } } \bigg)^2} + \varepsilon b{\int {\left| {\nabla {W_{\varepsilon ,Y}}} \right|} ^2}{\int {\left| {\nabla \varphi } \right|} ^2}-\left( {q - 1} \right)\int {Q\left( x \right)} W_{\varepsilon ,Y}^{q - 2}{\varphi ^2},\\
\end{aligned}
\end{equation}
and
\begin{equation}\label{33}
\begin{aligned}
{R_{\varepsilon,Y}}\left( \varphi  \right)
=&\ \frac{{\varepsilon b}}{4}{\bigg( {\int {{{\left| {\nabla \varphi } \right|}^2}} } \bigg)^2} + \varepsilon b\int {\nabla {W_{\varepsilon ,Y}}\nabla \varphi } {\int {\left| {\nabla \varphi } \right|} ^2} - \frac{1}{q}{\int {Q\left( x \right)\left( {{W_{\varepsilon ,Y}} + \varphi } \right)} ^q_+} \\
&\ +\frac{1}{q}\int {Q\left( x \right)W_{\varepsilon ,Y}^q}+\int Q(x)W_{\varepsilon,Y}^{q-1}\varphi+\frac{1}{2}(q-1)\int Q(x)W_{\varepsilon,Y}^{q-2} \varphi^2.
\end{aligned}
\end{equation}
In terms of\ $Q_{\varepsilon,Y}(\varphi)$,\ $\mathcal{L}_{\varepsilon,Y}:E_{\varepsilon,Y}^k\rightarrow E_{\varepsilon,Y}^k$ is a bounded linear mapping defined by:
\begin{equation}\label{399}
\begin{aligned}
  \langle \mathcal{L}_{\varepsilon,Y}\varphi_1, \varphi_2 \rangle_\varepsilon =&\ \langle\varphi_1 ,\varphi_2 \rangle_\varepsilon+\varepsilon b\int|\nabla W_{\varepsilon,Y}|^2\int \nabla\varphi_1 \nabla\varphi_2+2\varepsilon b\int\nabla W_{\varepsilon,Y}\nabla\varphi_1 \int\nabla W_{\varepsilon,Y}\nabla\varphi_2\\
  &\ -(q-1)\int Q(x)W_{\varepsilon,Y}^{q-2}\varphi_1 \varphi_2,\qquad \forall\ \varphi_1,\varphi_2\in E_{\varepsilon,Y}^k.
\end{aligned}
\end{equation}
The following result shows that $\mathcal{L}_{\varepsilon,Y}$ is invertible when restricted on\ $E_{\varepsilon,Y}^k$.
\begin{lem}\label{lm3.4}
There exist\ $\rho>0$, $\varepsilon_0>0$ and $\delta_0>0$, such that for every\ $\varepsilon\in(0,\varepsilon_0]$ and $\delta\in(0,\delta_0]$, there holds
\begin{equation*}
  \|\mathcal{L}_{\varepsilon,Y}\varphi \|_\varepsilon\ge \rho \|\varphi\|_\varepsilon,\qquad \forall\ \varphi\in E_{\varepsilon,Y}^k,
\end{equation*}
uniformly with respect to\ $Y\in D^k_{\varepsilon,\delta}$.
\end{lem}
\noindent
\begin{proof}
We use a contradiction argument. Assume that there exist\ $\varepsilon_n\rightarrow 0$, $\delta_{n}\rightarrow 0$, $ Y_n\in D^k_{\varepsilon_n,\delta_n}$ and $\varphi_n\in E_n\equiv E_{\varepsilon_n,Y_n}^k$ such that
\begin{equation}\label{338}
 \langle \mathcal{L}_{\varepsilon_n,Y_n}\varphi_n, h_n \rangle_{\varepsilon_n}=o_n(1)\|\varphi_n\|_{\varepsilon_n}\|h_n\|_{\varepsilon_n},\ \forall\ h_n\in E_n.
\end{equation}
With no loss of generality, we assume that\ $\|\varphi_n\|_{\varepsilon_n}^2=\varepsilon_n^3$, and denote
\begin{equation*}
\varphi_{n,{i_0}}(x)=\varphi_n(\varepsilon_n x+y_{n,{i_0}}),\ i_0=1,2,\ldots,k,
\end{equation*}
\begin{equation*}
 \tilde{E_n}=\{h_{n,{i_0}}(x):h_n(x) \in E_n,1\le i_0\le k \}.
\end{equation*}
Substituting\ (\ref{399}) into\ (\ref{338}), we obtain
\begin{equation}\label{339}
\begin{aligned}
   &\int\left(a\nabla \varphi_{n,{i_0}} \nabla h_{n,{i_0}}+ \varphi_{n,{i_0}} h_{n,{i_0}} \right)\\
   +&\ b\int\bigg(k|\nabla w|^2 +\sum_{i\not=j}\nabla w\Big(x+\frac{y_{n,{i_0}}-y_{n,{i}}}{\varepsilon_n}\Big) \nabla w\Big(x+\frac{y_{n,{i_0}}-y_{n,{j}}}{\varepsilon_n}\Big)\bigg)\int \nabla\varphi_{n,{i_0}} \nabla h_{n,{i_0}} \\
   +&\ 2b\int \sum_{i=1}^k \nabla w\Big(x+\frac{y_{n,{i_0}}-y_{n,{i}}}{\varepsilon_n}\Big)\nabla \varphi_{n,{i_0}}\int \sum_{i=1}^k \nabla w\Big(x+\frac{y_{n,{i_0}}-y_{n,{j}}}{\varepsilon_n}\Big)\nabla h_{n,{i_0}}\\
  -&\ (q-1)\int Q(\varepsilon_n x+y_{n,{i_0}})\bigg(  \sum_{i=1}^k w\Big(x+\frac{y_{n,{i_0}}-y_{n,{i}}}{\varepsilon_n}\Big) \bigg)^{q-2}\varphi_{n,{i_0}} h_{n,{i_0}}\\
   =&\ o_n(1)\bigg( \int\left( a|\nabla h_{n,{i_0}}|^2 +(h_{n,{i_0}})^2 \right)\bigg)^{\frac 12}.
\end{aligned}
\end{equation}
Since
\begin{equation*}
 \|\varphi_n\|_{\varepsilon_n}^2=\varepsilon_n^3\Rightarrow \int\big( a|\nabla \varphi_{n,{i_0}}|^2 +(\varphi_{n,{i_0}})^2 \big)=1,
\end{equation*}
we infer that $\{\varphi_{n,{i_0}}\}$ is a bounded sequence in\ $H^1(\mathbb R^3)$ for any $1\le i_0\le k$. Hence, up to a subsequence, there exists\ $\varphi\in H^1(\mathbb{R}^3)$ such that
\begin{equation*}
\begin{aligned}
  \varphi_{n,{i_0}} \rightharpoonup  \varphi &\quad \text{in\ $H^1(\mathbb R^3)$},\\
  \varphi_{n,{i_0}} \to \varphi  &\quad \text{in\ $L^p_{\mathrm {loc}}(\mathbb R^3)$,\ $1\le p<6$},\\
  \varphi_{n,{i_0}} \to \varphi   &\quad \text{a.e. in\ $\mathbb R^3$}.
\end{aligned}
\end{equation*}
Next, we will prove that\ $\varphi\equiv 0$. For any\ $i=1,\ldots,k$, $j=1,2,3$,
\begin{equation}\nonumber
   \int\bigg (\varepsilon_n^2 a \nabla\varphi_n \nabla\bigg(\frac{\partial w_{\varepsilon,y_i}}{\partial y_{ij}}\bigg)+\varphi_n \frac{\partial w_{\varepsilon,y_i}}{\partial y_{ij}}   \bigg) = 0
 \end{equation}
 is equivalent to
 \begin{equation}\label{341}
 \int\bigg(a\nabla\varphi_{n,{i_0}} \nabla\bigg(\frac{\partial w}{\partial x_j}\bigg)+\varphi_{n,{i_0}}\frac{\partial w}{\partial x_j}   \bigg)=0.
 \end{equation}
Thus, we can define an equivalent norm\ $\|u\|_1^2=\int (a|\nabla u|^2+u^2)$ in\ $H^1(\mathbb R^3)$, then
\begin{equation}\nonumber
  \varphi_n\in E_n
\end{equation}
is equivalent to
\begin{equation}\label{342}
\varphi_{n,{i_0}}\in(\ker \mathcal L)^{\perp}.
\end{equation}
Since\ $w$ is radially symmetric, we obtain
\begin{equation*}
  \left\langle \frac{\partial w}{\partial x_i}, \frac{\partial w}{\partial x_j} \right\rangle_1=0,\quad\forall\ i\not=j.
\end{equation*}
For every\ $h\in C_0^{\infty}(\mathbb R^3)$, define
\begin{equation}\label{343}
  h_{n,{i_0}}=h-\sum_{j=1}^3 a_{n,j} \frac{\partial w}{\partial x_j},
\end{equation}
where\ $a_{n,j}=\frac{\langle h,\partial_{x_j}w \rangle_1}{\langle \partial_{x_j}w , \partial_{x_j}w\rangle_1}$, then\ $h_n\in E_n$.
Substituting\ (\ref{343}) into\ (\ref{339}) and let\ $n\to\infty$, we obtain
\begin{equation*}
  \langle \mathcal L\varphi,h  \rangle -\Big\langle \mathcal L\varphi,\sum_{j=1}^3 a_{n,j} \partial_{x_j} w \Big\rangle=0.
\end{equation*}
Since\ $\partial_{x_j}w\in \mathrm{Ker}\,  \mathcal L$,
\begin{equation*}
  \langle \mathcal L\varphi,h  \rangle=0,\quad \forall\ h\in C_0^{\infty}(\mathbb R^3),
\end{equation*}
which implies that
\begin{equation}\label{344}
\varphi \in \mathrm{Ker}\,  \mathcal L.
\end{equation}
Note that $\varphi_n\in E_n $, letting\ $n\to \infty$ in\ (\ref{341}), we obtain
\begin{equation*}
  \int\bigg(a\nabla\varphi \nabla\bigg(\frac{\partial w}{\partial x_j}\bigg)+\varphi\frac{\partial w}{\partial x_j} \bigg)=0, \quad j=1,2,3.
\end{equation*}
Then
\begin{equation}\label{345}
  \varphi\in(\mathrm{Ker}\, \mathcal L)^{\perp}.
\end{equation}
Combining\ (\ref{344}) and (\ref{345}), we claim that\ $\varphi\equiv 0$.

Now we deduce contradiction. Note that\ $\varphi_{n,{i_0}}\rightarrow 0$ in\ $L^p_{\mathrm{loc}}(\mathbb R^3)\ (1\le p<6)$, so there exists\ $R>0$ sufficiently large such that
\begin{equation*}
  \int_{\mathbb R^3}w^{q-2}(x)(\varphi_{n,{i_0}})^2 = \int_{B_R(0)}w^{q-2}(x)(\varphi_{n,{i_0}})^2+\int_{B_R^c(0)}w^{q-2}(x)(\varphi_{n,{i_0}})^2
   = o_n(1)+o_R(1).
  \end{equation*}
Then
\begin{equation*}
\begin{aligned}
 &\bigg |(q-1)\int_{\mathbb R^3}Q(x)W_{\varepsilon_n,Y_n}^{q-2}\varphi_n^2\bigg| \\
 \le&\ C\varepsilon_n^3\int_{\mathbb R^3} \Bigg(\sum_{i=1}^k w\bigg(x+\frac{y_{n,{i_0}}-y_{n,{i}}}{\varepsilon_n}\bigg)\Bigg)^{q-2}(\varphi_{n,{i_0}})^2  \\
   \le&\ C\varepsilon_n^3\int_{\mathbb R^3}w^{q-2}(x)(\varphi_{n,{i_0}})^2+C\varepsilon_n^3\sum_{i\not=i_0}\int_{\mathbb R^3} w^{q-2}\bigg(x+\frac{y_{n,{i_0}}-y_{n,{i}}}{\varepsilon_n}\bigg)(\varphi_{n,{i_0}})^2  \\
   \le&\ \frac{1}{2}\varepsilon_n^3.
  \end{aligned}
  \end{equation*}
However,
\begin{equation*}
\begin{aligned}
  o_n(1)\|\varphi_n\|_{\varepsilon_n}^2 =&\ \langle \mathcal L_{\varepsilon_n,Y_n}\varphi_n,\varphi_n \rangle \\
   \ge&\ \|\varphi_n\|_{\varepsilon_n}^2-(q-1)\int_{\mathbb R^3}Q(y)W_{\varepsilon_n,Y_n}^{q-2}(y)\varphi_n^2(y)\\
   \ge&\ \frac{1}{2}\|\varphi_n\|_{\varepsilon_n}^2.
\end{aligned}
\end{equation*}
We reach a contradiction. The proof is complete.
\end{proof}

To apply contraction mapping principle to find a critical point of\ $J_{\varepsilon}(Y,\varphi)$, we first need to estimate\ $l_{\varepsilon,Y}(\varphi)$ and\ $R^{(i)}_{\varepsilon,Y}(\varphi)$ for $i=0,1,2$.
\begin{lem}\label{lm3.1}
There exists a constant\ $C>0$, independent
of\ $\varepsilon,\,\delta$, such that for any\ $Y\in D^k_{\varepsilon,\delta}$ and $\varphi  \in H_\varepsilon$, there holds
\begin{equation}\label{34}
  |l_{\varepsilon,Y}(\varphi)|\le C\varepsilon^{\frac32}\Big(\sum_{i=1}^k|Q(y_i)-Q(0)|+\sum_{i=1}^k\sum_{m=1}^{[h]}\varepsilon^m|D^mQ(y_i)| +\varepsilon^{[h]+1}+\sum_{i\not=j}e^{-\bar\theta\frac{\eta|y_i-y_j|}{\varepsilon}}\Big)\|\varphi\|_\varepsilon,
  \end{equation}
where\ $\bar \theta=\min\{\frac{q-1}{2},1\}$.
\end{lem}
\noindent
\begin{proof}
Since\ $w$ is the solution of\ (\ref{25}), we obtain that\ $w_{\varepsilon,y_i}\ (1\le i\le k)$ satisfies
\begin{equation*}
  -\bigg(\varepsilon^2a+\varepsilon bk\int|\nabla w_{\varepsilon,y_j}|^2\bigg)\Delta w_{\varepsilon,y_i}+w_{\varepsilon,y_i}=Q(0)w_{\varepsilon,y_i}^{q-1},\ j=1,\ldots,k.
\end{equation*}
We sum from\ $i=1$ to\ $k$ and get
\begin{equation*}
  -\bigg(\varepsilon^2a+\varepsilon bk\int|\nabla w_{\varepsilon,y_j}|^2\bigg)\Delta W_{\varepsilon,Y}+W_{\varepsilon,Y}=Q(0)\sum_{i=1}^kw_{\varepsilon,y_i}^{q-1}.
  \end{equation*}
Multiplying\ $\varphi$ on both sides of above equation and integrating, we obtain
\begin{equation*}
  \bigg(\varepsilon^2a+\varepsilon bk\int|\nabla w_{\varepsilon,y_j}|^2\bigg)\int\nabla W_{\varepsilon,Y}\nabla\varphi+\int W_{\varepsilon,Y}\varphi
=\int Q(0)\sum_{i=1}^kw_{\varepsilon,y_i}^{q-1}\varphi.
\end{equation*}
Then
\begin{equation}\label{35}
\begin{aligned}
  \langle  W_{\varepsilon,Y},\varphi\rangle_\varepsilon  &= \int\left(\varepsilon^2 a\nabla W_{\varepsilon,Y}\nabla\varphi +W_{\varepsilon,Y}\varphi\right)\\
   &= -\varepsilon bk\int|\nabla w_{\varepsilon,y_j}|^2\int \nabla W_{\varepsilon,Y}\nabla\varphi+\int Q(0)\sum_{i=1}^kw_{\varepsilon,y_i}^{q-1}\varphi.
\end{aligned}
\end{equation}
Substituting\ (\ref{35}) into\ (\ref{31}), we obtain
\begin{equation*}
\begin{aligned}
  l_{\varepsilon,Y}(\varphi) =&\ \varepsilon b\int \nabla W_{\varepsilon,Y}\nabla\varphi\bigg(\int |\nabla  W_{\varepsilon,Y}|^2-\int \sum_{j=1}^k|\nabla w_{\varepsilon,y_j}|^2\bigg)\\
  &\ -\bigg(\int Q(x)W_{\varepsilon,Y}^{q-1}\varphi-\int Q(0)\sum_{i=1}^k w_{\varepsilon,y_i}^{q-1}\varphi\bigg) \\
  =:&\ l_1-l_2.
  \end{aligned}
\end{equation*}
To estimate\ $l_1$, combining\ H$\ddot{\mathrm o}$lder inequality and\ \eqref{217} yields
\begin{equation}\label{622}
\begin{aligned}
  |l_1| &= \bigg|\varepsilon b\int \nabla W_{\varepsilon,Y}\nabla\varphi\bigg(\int |\nabla  W_{\varepsilon,Y}|^2-\int \sum_{j=1}^k|\nabla w_{\varepsilon,y_j}|^2\bigg)\bigg|\\
   &\le   \varepsilon b\bigg(\sum_{i=1}^{k}\int |\nabla w_{\varepsilon,y_i}\nabla \varphi|\bigg)\bigg(  \sum_{i\not=j}\int |\nabla w_{\varepsilon,y_i}\nabla w_{\varepsilon,y_j}|\bigg) \\
   &\le  \varepsilon^2 b \sum_{i\not=j}e^{-\frac{\eta|y_i-y_j|}{\varepsilon}} \bigg(\sum_{i=1}^k \|\nabla w_{\varepsilon,y_i}\|_{L^2}\|\nabla \varphi\|_{L^2}\bigg) \\
   &\le  \varepsilon^2 b \sum_{i\not=j}e^{-\frac{\eta|y_i-y_j|}{\varepsilon}}\bigg( k\varepsilon^{\frac 12}\|\nabla w\|_{L^2}\frac{1}{\sqrt{a}\varepsilon}\|\varphi\|_\varepsilon\bigg)\\
   &\le C\varepsilon^{\frac 32}\sum_{i\not=j}e^{-\frac{\eta|y_i-y_j|}{\varepsilon}}\|\varphi\|_\varepsilon.
\end{aligned}
\end{equation}
Next, we split\ $l_2$ into two parts:
\begin{equation}
\begin{aligned}
  l_2&= \int Q(x)\Big(\sum_{i=1}^k w_{\varepsilon ,y_i}\Big)^{q-1}\varphi-\sum_{i=1}^kQ(0)w_{\varepsilon,y_i}^{q-1}\varphi\\
    &= \int Q(x)\bigg(  \Big(\sum_{i=1}^k w_{\varepsilon ,y_i}\Big)^{q-1}-\sum_{i=1}^k w_{\varepsilon,y_i}^{q-1} \bigg)\varphi+\int\left(Q(x)-Q(0)\right) \sum_{i=1}^k w_{\varepsilon,y_i}^{q-1}\varphi\\
   &=: l_{21}+l_{22}.
\end{aligned}
\end{equation}
To estimate\ $l_{21}$, for\ $2<q\le3$, by the following inequality
\begin{equation*}
\begin{aligned}
  \left||a+b|^{q-1}-|a|^{q-1}-|b|^{q-1}\right|   &\le \left\{
  \begin{array}{ll}
    C|a||b|^{q-2}, & \hbox{if\ $|a|\le|b|$,} \\
    C|b||a|^{q-2}, & \hbox{if\ $|b|\le|a|$,}
  \end{array}
\right.\\
   &\le C|a|^{\frac{q-1}{2}}|b|^{\frac{q-1}{2}},
\end{aligned}
\end{equation*}
we obtain
\begin{equation}\label{36}
\begin{aligned}
  |l_{21}| &= \bigg|\int Q(x)\bigg(\Big(\sum_{i=1}^k w_{\varepsilon ,y_i}\Big)^{q-1}-\sum_{i=1}^k w_{\varepsilon,y_i}^{q-1}\bigg)\varphi\bigg| \\
   &\le C\int\sum_{i\not=j}w_{\varepsilon,y_i}^{\frac{q-1}{2}}w_{\varepsilon,y_j}^{\frac{q-1}{2}}|\varphi|
   \le C\sum_{i\not=j}\bigg(\int w_{\varepsilon ,y_i}^{q-1}w_{\varepsilon ,y_j}^{q-1}\bigg)^{\frac 12}\|\varphi\|_{L^2} \\
   &\le C\varepsilon^{\frac32}\sum_{i\not=j}e^{-\frac{q-1}{2}\frac{\eta|y_i-y_j|}{\varepsilon}}\|\varphi\|_\varepsilon.
\end{aligned}
\end{equation}
For\ $q>3$, we have
\begin{equation}\label{37}
\begin{aligned}
 |l_{21}| &= \bigg|\int Q(x)\bigg(\Big(\sum_{i=1}^k w_{\varepsilon ,y_i}\Big)^{q-1}-\sum_{i=1}^k w_{\varepsilon,y_i}^{q-1}\bigg)\varphi\bigg| \\
   &\le C\int\sum_{i\not=j}w_{\varepsilon ,y_i}^{q-2}w_{\varepsilon ,y_j}|\varphi|
   \le C\varepsilon^{\frac32}\sum_{i\not=j}e^{-\frac{\eta|y_i-y_j|}{\varepsilon}}\|\varphi\|_\varepsilon.
\end{aligned}
\end{equation}
Combining\ (\ref{36}) and (\ref{37}) yields
\begin{equation}\label{623}
  |l_{21}|\le C\varepsilon^{\frac32}\sum_{i\not=j}e^{-\min\{\frac{q-1}{2},1\}\frac{\eta|y_i-y_j|}{\varepsilon}}\|\varphi\|_\varepsilon.
\end{equation}
To estimate\ $l_{22}$, we split $l_{22}$ into two parts:
\begin{equation*}
\begin{aligned}
  |l_{22}| &= \int \sum_{i=1}^k \left(Q(x)-Q(0)\right) w_{\varepsilon,y_i}^{q-1}\varphi \\
   &= \int \sum_{i=1}^k \left(Q(x)-Q(y_i)\right) w_{\varepsilon,y_i}^{q-1}\varphi +\int \sum_{i=1}^k \left(Q(y_i)-Q(0)\right) w_{\varepsilon,y_i}^{q-1}\varphi\\
   &=: l_{221}+l_{222}.
   \end{aligned}
\end{equation*}
Estimating\ $l_{221}$, we have
\begin{equation}\label{624}
\begin{aligned}
 | l_{221}| \le &\ \sum_{i=1}^k \int |Q(x)-Q(y_i)| w_{\varepsilon,y_i}^{q-1}|\varphi|\\
   \le &\ \sum_{i=1}^k \bigg(\int |Q(x)-Q(y_i)|^2 w_{\varepsilon,y_i}^{2(q-1)}  \bigg)^{\frac 12}\|\varphi\|_{L^2} \\
   =&\ \sum_{i=1}^k \bigg( \int_{B_{\delta}(y_i)}|Q(x)-Q(y_i)|^2 w_{\varepsilon,y_i}^{2(q-1)}+\int_{B_{\delta}^c(y_i)}|(Q(x)-Q(y_i)|^2 w_{\varepsilon,y_i}^{2(q-1)}\bigg)^{\frac 12}\|\varphi\|_{L^2}\\
   \le&\ C\sum_{i=1}^k \bigg( \varepsilon^3\int_{|y|<\frac{\delta}{\varepsilon}}\Big(\sum_{m=1}^{[h]}\varepsilon^{2m}|D^mQ(y_i)|^2|y|^{2m}+\varepsilon^{2([h]+1)}|y|^{2 ([h]+1)} \Big)w^{2(q-1)}(y)\bigg)^{\frac 12}\|\varphi\|_{L^2}\\
&\ +C\sum_{i=1}^k \bigg(\varepsilon^3 \int_{|y|\ge\frac{\delta}{\varepsilon}}w^{2(q-1)}(y)\bigg)^{\frac 12}\|\varphi\|_{L^2} \\
 \le&\ C\varepsilon^{\frac 32}\sum_{i=1}^k\Big(C_1\sum_{m=1}^{[h]}\varepsilon^{m}|D^mQ(y_i)|+C_2\varepsilon^{[h]+1}+C_3e^{-(q-1)\frac{\eta\delta}{2\varepsilon}}\Big)\|\varphi\|_{L^2}\nonumber\\
\le&\ C\varepsilon^{\frac32}\Big(\sum_{i=1}^k\sum_{m=1}^{[h]}\varepsilon^m|D^mQ(y_i)|+\varepsilon^{[h]+1}\Big)\|\varphi\|_\varepsilon.
\end{aligned}
\end{equation}
Finally, to estimate\ $l_{222}$, combining\ H$\ddot{\mathrm o}$lder inequality and Lemma\ \ref{lm2.1}, we obtain
\begin{equation}\label{625}
\begin{aligned}
  |l_{222}| &=\left| \int \sum_{i=1}^k \left(Q(y_i)-Q(0)\right) w_{\varepsilon,y_i}^{q-1}\varphi\right|\\
   &\le \sum_{i=1}^k \left|Q(y_i)-Q(0)\right|\int w_{\varepsilon,y_i}^{q-1}|\varphi|
   \le \sum_{i=1}^k\left|Q(y_i)-Q(0)\right|\bigg(\int w_{\varepsilon,y_i}^{q}\bigg)^{\frac{q-1}{q}}\bigg(\int |\varphi|^q\bigg)^{\frac 1q} \\
   &\le C\varepsilon^{\frac 32}\sum_{i=1}^k \left|Q(y_i)-Q(0)\right|\|\varphi\|_\varepsilon.
\end{aligned}
\end{equation}
Combining\ (\ref{622}) and\ (\ref{623})--(\ref{625}) yields\ (\ref{34}).
\end{proof}

\begin{lem}\label{lm3.2}
There exists a constant\ $C>0$, independent
of\ $\varepsilon,\,\delta$, such that for any\ $\varphi\in H_{\varepsilon}$, there holds
\begin{equation}\label{39}
 \| R_{\varepsilon,Y}^{(i)}(\varphi)\|\le Cb\varepsilon^{-\frac 32}\big(1+\varepsilon^{-\frac 32}\|\varphi\|_\varepsilon\big)\|\varphi\|_\varepsilon^{3-i}+C\varepsilon^{-\frac{3(q-2)}{2}}\|\varphi\|_{\varepsilon}^{q-i},\ i=0,1,2.
\end{equation}
\end{lem}
\noindent
\begin{proof}
By\ (\ref{33}),
\begin{equation*}
  R_{\varepsilon,Y}(\varphi)=A_1(\varphi)-A_2(\varphi),
\end{equation*}
where
\begin{equation}\label{310}
     A_1(\varphi) = \frac{{\varepsilon b}}{4}{\bigg( {\int {{{\left| {\nabla \varphi } \right|}^2}} } \bigg)^2} + \varepsilon b\int {\nabla {W_{\varepsilon ,Y}}\nabla \varphi } {\int {\left| {\nabla \varphi } \right|} ^2},
     \end{equation}
     and
     \begin{equation}\label{311}
     A_2(\varphi) = \frac{1}{q}\int Q(x)\Big( (W_{\varepsilon,Y}+\varphi)^q_+-W^q_{\varepsilon,Y} - qW_{\varepsilon,Y}^{q-1}\varphi-\frac{q(q-1)}{2} W_{\varepsilon,Y}^{q-2}\varphi^2 \Big).
   \end{equation}
For\ $\forall \ \psi,\xi\in H_\varepsilon$, we obtain
\begin{equation}\label{312}
\begin{aligned}
   \big\langle A_1^{(1)}(\varphi),\psi\big\rangle =&\ \varepsilon b\bigg(\int|\nabla\varphi|^2 \int\nabla\varphi\nabla\psi+ \int|\nabla\varphi|^2 \int\nabla W_{\varepsilon,Y}\nabla\psi\bigg) \\
   &\ +2\varepsilon b\int\nabla W_{\varepsilon,Y}\nabla\varphi\int\nabla\varphi\nabla\psi,\\
   \end{aligned}
   \end{equation}
   \begin{equation}\label{313}
\big\langle A_2^{(1)}(\varphi),\psi\big\rangle =\int Q(x)\big((W_{\varepsilon,Y}+\varphi)^{q-1}_+-W_{\varepsilon,Y}^{q-1}-(q-1)W_{\varepsilon,Y}^{q-2}\varphi \big)\psi,
 \end{equation}
  \begin{equation}\label{314}
   \begin{aligned}
\big\langle A_1^{(2)}(\varphi)[\psi],\xi\big\rangle =&\ \varepsilon b \bigg(2\int \nabla\varphi \nabla\xi \int \nabla\varphi\nabla\psi+\int |\nabla \varphi|^2 \int\nabla\xi\nabla\psi\bigg)\\
    &\ +2\varepsilon b\bigg( \int \nabla W_{\varepsilon,Y}\nabla\xi \int \nabla\varphi \nabla\psi+\int \nabla W_{\varepsilon,Y}\nabla\varphi \int \nabla\xi \nabla\psi      \bigg)\\
  &\ +2\varepsilon b\int \nabla\varphi \nabla\xi \int \nabla W_{\varepsilon,Y}\nabla\psi,
  \end{aligned}
  \end{equation}
   \begin{equation}\label{315}
  \big\langle A_2^{(2)}(\varphi)[\psi],\xi\big\rangle = (q-1) \int Q(x)\big( (W_{\varepsilon,Y}+\varphi)^{q-2}_+- W_{\varepsilon,Y}^{q-2} \big)\psi\xi.
\end{equation}
Next, we estimate\ $A_1^{(i)}(\varphi),\ i=0,1,2$. Note that
\begin{equation}\label{316}
  \|\nabla\varphi\|_{L^2}\le\frac{1}{\sqrt{a}\varepsilon}\|\varphi\|_\varepsilon,
\end{equation}
and
\begin{equation*}
  \int |\nabla W_{\varepsilon,Y}|^2 \le k\sum_{i=1}^k\int|\nabla w_{\varepsilon,y_i}|^2 = k^2\varepsilon\int |\nabla w|^2,
   \end{equation*}
we have
\begin{equation}\label{317}
  \|\nabla W_{\varepsilon,Y}\|_{L^2}\le C_1 \varepsilon^{\frac 12},
  \end{equation}
where $C_1=k\|\nabla w\|_{L^2}$. Combining\ (\ref{316}) and (\ref{317}), we obtain that for\ $\forall \ \psi,\xi,\nu\in H_\varepsilon$, there hold
\begin{equation}\label{318}
  \int |\nabla\varphi\nabla\psi |\int| \nabla W_{\varepsilon,Y}\nabla\xi |\le C\varepsilon^{-\frac 52} \|\varphi\|_\varepsilon \|\psi\|_\varepsilon \|\xi\|_\varepsilon  ,   \\
  \end{equation}
  \begin{equation}\label{319}
   \int |\nabla\varphi\nabla\psi |\int |\nabla\nu \nabla\xi |\le C\varepsilon^{-4} \|\varphi\|_\varepsilon \|\psi\|_\varepsilon \|\nu\|_\varepsilon \|\xi\|_\varepsilon.
\end{equation}
Combining\ (\ref{310}), (\ref{312}), (\ref{314}), (\ref{318}) and (\ref{319}) yields
\begin{equation}\label{320}
\begin{aligned}
  \|A_1^{(i)}(\varphi)\| &\le Cb\varepsilon^{-\frac 32}\|\varphi\|_\varepsilon^{3-i}+Cb\varepsilon^{-3}\|\varphi\|_\varepsilon^{4-i}  \\
   &\le Cb\varepsilon^{-\frac 32}\|\varphi\|_\varepsilon^{3-i}\big(\varepsilon^{-\frac 32}\|\varphi\|_\varepsilon+1\big).
\end{aligned}
\end{equation}
Then, we estimate\ $A_2^{(i)}(\varphi),\ i=0,1,2$. For~$2<q\le3$, we apply the following elementary inequalities: for~$e,f\in\mathbb{ R}$~, there exist constants~$C_1(q),~C_2(q),~C_3(q)>0$ such that
\begin{equation*}
|(e+f)_+^{q}-e_+^{q}-qe_+^{q-1}f-\frac{q(q-1)}{2}e_+^{q-2}f^2|\leq C_1(q)|f|^{q},
\end{equation*}
\begin{equation*}
|(e+f)_+^{q-1}-e_+^{q-1}-(q-1)e_+^{q-2}f|\leq C_2(q)|f|^{q-1},
\end{equation*}
and
\begin{equation*}
|(e+f)_+^{q-2}-e_+^{q-2}|\leq C_3(q)|f|^{q-2}.
\end{equation*}
Combining the above inequalities and Lemma\ \ref{lm2.1} yields
\begin{equation}\label{321}
|A_2(\varphi)|\le C\varepsilon^{-\frac{3(q-2)}{2}}\|\varphi\|_\varepsilon^{q},
\end{equation}
\begin{equation}\label{322}
\|A_2^{(1)}(\varphi)\|\le C\varepsilon^{-\frac{3(q-2)}{2}}\|\varphi\|_\varepsilon^{q-1},
\end{equation}
\begin{equation}\label{323}
\|A_2^{(2)}(\varphi)\|\le C\varepsilon^{-\frac{3(q-2)}{2}}\|\varphi\|_\varepsilon^{q-2}.
\end{equation}
Similarly, for~$3<q<6$ and $e,f\in\mathbb{ R}$, there exist constants~$C'_1(q),~C'_2(q),~C'_3(q)>0$ such that
\begin{equation*}
|(e+f)_+^{q}-e_+^{q}-qe_+^{q-1}f-\frac{q(q-1)}{2}e_+^{q-2}f^2|\leq C'_1(q)(|e|^{q-3}+|f|^{q-3})|f|^3
\end{equation*}
\begin{equation*}
|(e+f)_+^{q-1}-e_+^{q-1}-(q-1)e_+^{q-2}f|\leq C'_2(q)(|e|^{q-3}+|f|^{q-3})|f|^2,
\end{equation*}
and
\begin{equation*}
|(e+f)_+^{q-2}-e_+^{q-2}|\leq C'_3(q)(|e|^{q-3}+|f|^{q-3})|f|.
\end{equation*}
Combining the above inequalities and Lemma\ \ref{lm2.1} yields
\begin{equation}\label{600}
\begin{aligned}
|A_2(\varphi)|&\le C'_1(q)\int\big(|W_{\varepsilon,Y}|^{q-3}+|\varphi|^{q-3}\big)|\varphi|^3\\
&\le C\bigg(\int|W_{\varepsilon,Y}|^{2(q-3)}\bigg)^\frac{1}{2}\varepsilon^{-3}\|\varphi\|_\varepsilon^3+ C\varepsilon^{-\frac{3(q-2)}{2}}\|\varphi\|_\varepsilon^{q}\\
&\le C\big(\varepsilon^{-\frac{3}{2}}\|\varphi\|_\varepsilon^3+\varepsilon^{-\frac{3(q-2)}{2}}\|\varphi\|_\varepsilon^{q}\big).
\end{aligned}
\end{equation}
By the same token, we obtain
\begin{equation}\label{601}
\|A_2^{(1)}(\varphi)\|\le C\big(\varepsilon^{-\frac{3}{2}}\|\varphi\|_\varepsilon^2+\varepsilon^{-\frac{3(q-2)}{2}}\|\varphi\|_\varepsilon^{q-1}\big),
\end{equation}
\begin{equation}\label{602}
\|A_2^{(2)}(\varphi)\|\le C\big(\varepsilon^{-\frac{3}{2}}\|\varphi\|_\varepsilon+\varepsilon^{-\frac{3(q-2)}{2}}\|\varphi\|_\varepsilon^{q-2}\big).
\end{equation}
Combining\ (\ref{320})--(\ref{602}) yields\ (\ref{39}).
\end{proof}

To state the lemma that follows, we define
\begin{equation}\label{500}
\begin{split}
N_{\varepsilon}=\bigg\{ \varphi\in E_{\varepsilon,Y}^k:\|\varphi\|_{\varepsilon}\le& \varepsilon^{\frac32} \Big(\sum_{i=1}^k\left|Q(y_i)-Q(0)\right|^{1-\tau}+\sum_{i=1}^k\sum_{m=1}^{[h]}\varepsilon^{m-\tau}|D^mQ(y_i)|^{1-\tau} \\
&+\varepsilon^{[h]+1-\tau}+\sum_{i\not=j}e^{-(\bar{\theta}-\tau)\frac{\eta|y_i-y_j|}{\varepsilon}}\Big)\bigg\},
\end{split}
\end{equation}
where\ $0<\tau<\min\{1,\bar\theta \}$.
\begin{lem}\label{lm3.5}
There exist $\varepsilon_0,\ \delta_0$ sufficiently small such that for every\ $\varepsilon\in(0,\varepsilon_0]$ and $\delta\in(0,\delta_0]$, there exists a\ $C^1$ map\ $\varphi_{\varepsilon}:D_{\varepsilon,\delta}^k \to N_\varepsilon;\ Y \mapsto  \varphi_{\varepsilon,Y}$ satisfying
\begin{equation}\label{400}
  \left\langle  \frac{\partial J_{\varepsilon} (Y,\varphi_{\varepsilon,Y})   }{\partial \varphi} ,\psi \right\rangle=0,\quad\forall\ \psi\in H_\varepsilon,\ \forall\ Y\in D_{\varepsilon,\delta}^k.
\end{equation}
Moreover, we can choose\ $0<\tau<\min\{1,\bar\theta \}$ sufficiently small, such that
\begin{equation}\label{401}
\|\varphi_{\varepsilon,Y}\|_{\varepsilon}\le \varepsilon^{\frac32} \Big(\sum_{i=1}^k\left|Q(y_i)-Q(0)\right|^{1-\tau}+\sum_{i=1}^k\sum_{m=1}^{[h]}\varepsilon^{m-\tau}|D^mQ(y_i)|^{1-\tau}+\varepsilon^{[h]+1-\tau}+\sum_{i\not=j}e^{-(\bar{\theta}-\tau)\frac{\eta|y_i-y_j|}{\varepsilon}}\Big).
\end{equation}
\end{lem}
\noindent
\begin{proof}
Recall that
\begin{equation*}
  J_{\varepsilon} (Y,\varphi)=J_{\varepsilon} (Y,0)+\langle I_\varepsilon^{'}(W_{\varepsilon,Y}) ,\varphi \rangle+\frac 12 \langle I_\varepsilon^{''}(W_{\varepsilon,Y})[\varphi],\varphi  \rangle+R_{\varepsilon,Y}(\varphi),
\end{equation*}
so we have
\begin{equation*}
 \left\langle \frac{\partial J_\varepsilon}{\partial\varphi} ,\psi  \right\rangle=\langle I_\varepsilon^{'}(W_{\varepsilon,Y}) ,\psi \rangle+\langle I_\varepsilon^{''}(W_{\varepsilon,Y})[\varphi],\psi  \rangle+\langle R_{\varepsilon,Y}^{'}(\varphi) ,\psi \rangle,\ \forall\ \psi\in H_\varepsilon,
\end{equation*}
i.e.
\begin{equation}\label{346}
\begin{aligned}
  \frac{\partial J_\varepsilon}{\partial\varphi} &= I_\varepsilon^{'}(W_{\varepsilon,Y})+I_\varepsilon^{''}(W_{\varepsilon,Y})[\varphi]+ R_{\varepsilon,Y}^{'}(\varphi)\\
   &= l_{\varepsilon,Y}+I_\varepsilon^{''}(W_{\varepsilon,Y})[\varphi]+ R_{\varepsilon,Y}^{'}(\varphi).
\end{aligned}
\end{equation}
Then $\frac{\partial J_\varepsilon}{\partial\varphi}$ is a bounded linear functional in\ $N_\varepsilon$. Denote
\begin{equation*}
  \mathfrak{W}=\{f: f \text{ is a bounded linear functional defined on\ $H_\varepsilon$}\}.
\end{equation*}
For\ $\forall f\in \mathfrak{W}$, by Riesz representation theorem, there exists a unique\ $\hat{f}\in H_\varepsilon$ such that
\begin{equation*}
  f(\psi)= \langle \hat{f},\psi  \rangle_\varepsilon,\quad \forall\ \psi\in H_\varepsilon.
\end{equation*}
So we can define a map\ $\sigma:\mathfrak{W}\to H_\varepsilon;\ f\mapsto \hat{f} $.

Let\ $\mathfrak{W}^*=\sigma (\mathfrak{W})$.
 Next, we prove\ $\sigma$ is a linear isomorphic map from\ $\mathfrak{W}$ to $\mathfrak{W}^*$.
In fact, if\ $\sigma(f_1)=\sigma(f_2)$, in the sense that\ $\hat{f_1}=\hat{f_2}$, we obtain
\begin{equation*}
   f_1(\psi)=\langle \hat{f_1},\psi  \rangle_\varepsilon=\langle \hat{f_2},\psi  \rangle_\varepsilon=f_2(\psi),\quad \forall\ \psi\in H_\varepsilon.
\end{equation*}
Then\ $f_1=f_2$ and $\sigma$ is injective.  Besides, for\ $\forall f_1,f_2\in \mathfrak{W}$,
\begin{equation*}
  \langle \widehat{f_1+f_2},\psi   \rangle_\varepsilon=(f_1+f_2)(\psi)=f_1(\psi)+f_2(\psi)=\langle \hat{f_1} ,\psi  \rangle_\varepsilon+\langle \hat{f_2},\psi    \rangle_\varepsilon=\langle \hat{f_1}+\hat{f_2}    ,\psi \rangle_\varepsilon,
\end{equation*}
which implies\ $\widehat{f_1+f_2}=\hat{f_1}+\hat{f_2}$, in the sense that $\sigma(f_1+f_2)=\sigma(f_1)+\sigma(f_2)$.\\
And for\ $\forall k \in\mathbb R,\ f\in \mathfrak{W}$, we obtain
\begin{equation*}
  \langle \widehat{kf},\psi   \rangle_\varepsilon=(kf)(\psi) =kf(\psi)  = k \langle  \hat f  ,\psi   \rangle_\varepsilon= \langle   k\hat f ,\psi   \rangle_\varepsilon.
\end{equation*}
Thus,\ $\widehat{kf}=k\hat{f}$ and\ $\sigma(kf)=k\sigma(f)$.

Therefore,\ (\ref{346}) is equivalent to
\begin{equation}\label{347}
  \hat{\frac{\partial J_\varepsilon}{\partial\varphi}}=\hat l_{\varepsilon,Y}+\mathcal L_{\varepsilon,Y}(\varphi)+\hat R_{\varepsilon,Y}^{'}(\varphi).
\end{equation}
Since $\mathcal L_{\varepsilon,Y}$ is invertible in $E_{\varepsilon,Y}^k$ by Lemma \ref{lm3.4}, it is sufficient to find $\varphi\in N_\varepsilon$ that satisfies
\begin{equation}\label{349}
 \varphi=-\mathcal L_{\varepsilon,Y}^{-1}(\hat l_{\varepsilon,Y})-\mathcal L_{\varepsilon,Y}^{-1}(\hat R_{\varepsilon,Y}^{'}(\varphi))=:\mathcal A_{\varepsilon,Y}(\varphi).
\end{equation}
Next, We prove that\ $\mathcal A_{\varepsilon,Y}$ is a contraction map on\ $N_\varepsilon$. First, for\ $\forall \varphi\in N_\varepsilon$, we have
\begin{equation}\label{350}
\begin{aligned}
 \|\mathcal A_{\varepsilon,Y}(\varphi)\|_\varepsilon &\le \frac 1\rho \|   \hat l_{\varepsilon,Y}  \|_\varepsilon+\frac 1\rho \|  \hat R_{\varepsilon,Y}^{'}(\varphi)   \|_\varepsilon \\
   &= \frac 1\rho \|    l_{\varepsilon,Y}  \|+\frac 1\rho \|   R_{\varepsilon,Y}^{'}(\varphi)   \|.
\end{aligned}
\end{equation}
By Lemma\ \ref{lm3.1}, we obtain
\begin{equation*}
  \|l_{\varepsilon,Y}\|\le C\varepsilon^{\frac32}\Big(\varepsilon^{[h]+1}+\sum_{i=1}^k|Q(y_i)-Q(0)|+\sum_{i=1}^k\sum_{m=1}^{[h]}\varepsilon^m|D^mQ(y_i)|+\sum_{i\not=j}e^{-\bar\theta\frac{\eta|y_i-y_j|}{\varepsilon}}\Big).  \end{equation*}
Choose\ $\varepsilon$, $\delta$ sufficiently small such that
\begin{equation*}
  \left\{
     \begin{array}{ll}
       C\varepsilon^{\tau} < \frac \rho 2, &  \\
       C|Q(y_i)-Q(0)|^{\tau} < \frac \rho 2, & \hbox{$i=1,\ldots,k$,} \\
       C\varepsilon^{\tau}|D^mQ(y_i)|^{\tau} < \frac \rho 2, & \hbox{$i=1,\ldots,k$,}\ m=1,\ldots,[h],\\
       Ce^{-\tau\frac{\eta|y^i-y^j|}{\varepsilon}}< \frac \rho 2, & \hbox{$i\not=j$,}
     \end{array}
   \right.
\end{equation*}
Then
\begin{equation}\label{351}
\begin{aligned}
   \|l_{\varepsilon,Y}\|\le&\ \frac {\rho}2 \varepsilon^{\frac32}\Big(\sum_{i=1}^k|Q(y_i)-Q(0)|^{1-\tau}+\sum_{i=1}^k\sum_{m=1}^{[h]}\varepsilon^{m-\tau}|D^mQ(y_i)|^{1-\tau}\Big)\\
   &\
   +\frac {\rho}2 \varepsilon^{\frac32}\Big(\varepsilon^{[h]+1-\tau}+\sum_{i\not=j}e^{-(\bar\theta-\tau)\frac{\eta|y_i-y_j|}{\varepsilon}}\Big).
\end{aligned}
\end{equation}
As\ $\varphi\in N_\varepsilon$,
\begin{equation*}
  \varepsilon^{-\frac 32}\|\varphi\|_\varepsilon=o_{\varepsilon}(1)+o_\delta(1).
\end{equation*}
So for\ $\varepsilon$, $\delta$ sufficiently small, by Lemma\ \ref{lm3.2}, we have
\begin{equation}\label{352}
 \|R_{\varepsilon,Y}^{'}(\varphi)\|=\left(o_{\varepsilon}(1)+o_\delta(1)\right)\|\varphi\|_\varepsilon \le \frac \rho 2\|\varphi\|_\varepsilon.
\end{equation}
Combining\ (\ref{350})--(\ref{352}) yields
\begin{equation*}
\begin{aligned}
  \|\mathcal A_{\varepsilon,Y}(\varphi)\|_\varepsilon \le&\ \varepsilon^{\frac32}\Big(\sum_{i=1}^k|Q(y_i)-Q(0)|^{1-\tau}+\sum_{i=1}^k\sum_{m=1}^{[h]}\varepsilon^{m-\tau}|D^mQ(y_i)|^{1-\tau}\Big)\nonumber\\
   &\ +\varepsilon^{\frac32}\Big(\varepsilon^{[h]+1-\tau}+\sum_{i\not=j}e^{-(\bar\theta-\tau)\frac{\eta|y_i-y_j|}{\varepsilon}}\Big).
\end{aligned}
\end{equation*}
Hence, \ $\mathcal A_{\varepsilon,Y}(N_\varepsilon)\subset N_\varepsilon$. On the other hand, for every\ $\varphi,\psi\in N_\epsilon$,
\begin{equation*}
\begin{aligned}
   \|\mathcal A_{\varepsilon,Y}(\varphi)-\mathcal A_{\varepsilon,Y}(\psi)\|_\varepsilon  &= \|\mathcal L_{\varepsilon,Y}^{-1}\big(\hat R_{\varepsilon,Y}^{'}(\varphi)\big)-\mathcal L_{\varepsilon,Y}^{-1}\big(\hat R_{\varepsilon,Y}^{'}(\psi)\big)\|_\varepsilon \\
  &\le \frac 1\rho\| R_{\varepsilon,Y}^{'}(\varphi)- R_{\varepsilon,Y}^{'}(\psi) \| \\
  &= \frac 1\rho \| R_{\varepsilon,Y}^{''}(\xi \varphi+(1-\xi)\psi)\|\| \varphi-\psi\|_\varepsilon,\quad 0<\xi<1.
\end{aligned}
\end{equation*}
By Lemma\ \ref{lm3.2}, we obtain
\begin{equation*}
\begin{aligned}
  \| R_{\varepsilon,Y}^{''}(\xi \varphi+(1-\xi)\psi)\| \le&\ C\varepsilon^{-\frac{3(q-2)}{2}}\|\xi \varphi+(1-\xi)\psi\|_\varepsilon^{q-2}\nonumber\\
  &\ +Cb \varepsilon^{-\frac32}\big(1+\varepsilon^{-\frac32}\|\xi \varphi+(1-\xi)\psi\|_\varepsilon\big) \|\xi \varphi+(1-\xi)\psi\|_\varepsilon\\
   =&\ o_\varepsilon(1).
\end{aligned}
\end{equation*}
Thus, for\ $\varepsilon$ sufficiently small, we have
\begin{equation*}
 \|\mathcal A_{\varepsilon,Y}(\varphi)-\mathcal A_{\varepsilon,Y}(\psi)\|_\varepsilon \le \frac 12 \|\varphi-\psi\|_\varepsilon.
\end{equation*}
So\ $\mathcal A_{\varepsilon,Y}$ is a contraction map on\ $N_\varepsilon$. By contraction mapping principle, we infer that\ (\ref{349}) has a unique solution. Finally, by similar arguments as that
of Cao, Noussair and Yan\ \cite{6}, we can deduce that\ $\varphi_\varepsilon$ belongs to\ $C^1$.
\end{proof}

\section{Proof of Theorem\ \ref{thm1.1}}
In this section, without loss of generality, we assume\ $Q(0)=1$. By Lemma\ \ref{lm3.5}, we can define a \ $C^1$ function on\ $D^k_{\varepsilon.\delta}$, in the sense that
\begin{equation*}
  K(Y)=:J_\varepsilon(Y,\varphi_{\varepsilon,Y}),\ Y\in D^k_{\varepsilon,\delta}.
\end{equation*}

Define
\begin{equation*}
 c_{\varepsilon,1}=\varepsilon^3(kA-k^2B-T\varepsilon^{\alpha h_1}),\quad c_{\varepsilon,2}=\varepsilon^3(kA-k^2B+\mu),
\end{equation*}
where\ $A=\frac{q-2}{2q}\|w\|_{L^q}^q,\ B=\frac{b}{4}\|\nabla w\|_{L^2}^4$,\ $\mu,\ T$ are positive constants,\ $\varepsilon^\alpha \le \frac{\delta}{2}$ and $\alpha\in (0,1)$ is a fixed constant close to\ $1$.

Denote
\begin{equation*}
{\Omega _\gamma }=\Big\{Y = ({y_1}, \ldots ,{y_k}):\ {y_i} \in B_\delta ^t(0) \times B_\gamma ^{3 - t}(0),\ i = 1, \ldots ,k,\ \frac{{\eta |{y_i} - {y_j}|}}{\varepsilon } \ge {R_1},\ i \ne j\Big\},
\end{equation*}
where\ $B_\tau^l(0)=\{y\in \R^l:\ |y|\le \tau\}$,\ $R_1>0 $ is a large constant, and
\begin{equation*}
K^c=\big\{Y:\ Y\in \Omega_{\varepsilon^\alpha},\ K(Y)\le c\big\}.
\end{equation*}

\begin{lem}\label{lm3.6}
 For any\ $\varphi\in E^k_{\varepsilon,Y}$, there holds
 \begin{equation}\label{360}
  \langle \mathcal {L}_{\varepsilon,Y}\varphi,\varphi \rangle_\varepsilon=O(\|\varphi\|_\varepsilon^2).
 \end{equation}
\end{lem}

\noindent
\begin{proof}
By the definition of\ $\mathcal {L}_{\varepsilon,Y}$, we have
\begin{equation}\label{361}
\begin{aligned}
  \langle \mathcal {L}_{\varepsilon,Y}\varphi,\varphi \rangle_\varepsilon
   =&\ \langle \varphi,\varphi    \rangle_\varepsilon+\varepsilon b\int|\nabla W_{\varepsilon,Y}|^2\int|\nabla \varphi|^2\\
   &\ +2\varepsilon b\bigg( \int \nabla W_{\varepsilon,Y} \nabla \varphi \bigg)^2-(q-1)\int Q(x)W_{\varepsilon,Y}^{q-2}\varphi^2.
\end{aligned}
\end{equation}
Calculating directly yields
\begin{equation}\label{362}
 \varepsilon b\int|\nabla W_{\varepsilon,Y}|^2\int|\nabla \varphi|^2 \le\varepsilon bk\int\sum_{i=1}^k|\nabla w_{\varepsilon,y_i}|^2\int|\nabla \varphi|^2
   \le C\|\varphi\|_\varepsilon^2.
 \end{equation}
By\ H\"older inequality, we obtain
\begin{equation}\label{363}
 \varepsilon b\bigg( \int \nabla W_{\varepsilon,Y} \nabla \varphi \bigg)^2 \le \varepsilon b\int|\nabla W_{\varepsilon,Y}|^2\int|\nabla \varphi|^2
 \le C\|\varphi\|_\varepsilon^2.
\end{equation}
Finally, as\ $Q(x)$ is bounded, we have
\begin{equation}\label{364}
 \int Q(x)W_{\varepsilon,Y}^{q-2}\varphi^2 \le C\bigg(\int W_{\varepsilon,Y}^{q}\bigg)^{\frac{q-2}{q}}\bigg(\int|\varphi|^{q}\bigg)^{\frac{2}{q}}
   \le C\|\varphi\|_\varepsilon^2.
\end{equation}
Combining\ (\ref{361})--(\ref{364}) yields\ (\ref{360}).
\end{proof}

\begin{lem}\label{lm3.7}
 There exist constants\ $\varepsilon_0,\delta_0>0$, such that for any\ $\varepsilon\in(0,\varepsilon_0]$ and $\delta\in(0,\delta_0]$, $(Y,\varphi)$ is a critical point of\ $J_\varepsilon$ on\ $D^k_{\varepsilon,\delta}\times E^k_{\varepsilon,Y}$ is equivalent to
 \begin{equation*}
   u=\sum_{i=1}^k w_{\varepsilon,y_i}+\varphi
 \end{equation*}
 is a critical point of\ $I_\varepsilon$.
\end{lem}
\begin{proof}
This lemma can be proved by the same arguments as that of\ \cite{3,6} with minor modifications. We omit the details.
\end{proof}

\begin{lem}\label{lm4.1}
For every\ $Y\in \partial \Omega_{\varepsilon^\alpha}$, we have either\ $K(Y)<c_{\varepsilon,1}$ or\ $\frac {\partial K(Y)} {\partial n}>0$, where\ $n$  is the outward unit normal of\ $\partial \Omega_{\varepsilon^{\alpha}}$ at\ $Y$.
\end{lem}

\noindent
\begin{proof}
We divide the proof of this lemma into two steps.

Step 1: Suppose that\ $\frac{{\eta |{y_i} - {y_j}|}}{\varepsilon }={R_1}$ for some\ $i\ne j$, or\ ${y_i} \in \partial B_\delta ^t(0) \times B_{\varepsilon^\alpha} ^{3 - t}(0)$ for some\ $i\in \{1,\ldots,k\}$. We claim that\ $Y\in K^{c_{\varepsilon,1}}$.

In fact, since\ $\varphi_{\varepsilon,Y}\in N_\varepsilon $, by Lemma\ \ref{lm3.1}, we obtain
\begin{equation}\label{41}
  |l_{\varepsilon,Y}(\varphi_{\varepsilon,Y})|=O(\|\varphi_{\varepsilon,Y}\|_\varepsilon^2).
\end{equation}
And by Lemma\ \ref{lm3.2}, we have
\begin{equation}\label{42}
  |R_{\varepsilon,Y}(\varphi_{\varepsilon,Y})|=o_\varepsilon(1)\|\varphi_{\varepsilon,Y}\|_\varepsilon^2.
\end{equation}
Combining\ (\ref{120}), (\ref{360}), (\ref{41}) and (\ref{42}) yields
\begin{equation}\label{43}
  J_\varepsilon(Y,{\varphi_{\varepsilon,Y}})=J_\varepsilon(Y,0)+O(\|\varphi_{\varepsilon,Y}\|_\varepsilon^2).
\end{equation}
Then combining \ref{lemma A.1},\ (\ref{43}) and\ (\ref{401}) yields
\begin{equation}\label{44}
\begin{aligned}
   K(Y) =&\ \varepsilon^3\left(kA-k^2B\right)-\frac{2}{q-2}A\varepsilon^3 \sum_{i=1}^k \left(Q(y_i)-1\right)\\
   &\
   -\int \sum_{i=1}^{k-1}w_{\varepsilon,y_i}\Big(\sum_{j=i+1}^kw_{\varepsilon,y_j}\Big)^{q-1}+O\big(\varepsilon^{4+[h]}\big)\\   &\ +O\Big(\sum_{i=1}^k\sum_{m=1}^{[h]}\varepsilon^{3+m}|D^mQ(y_i)|+\varepsilon^3 \sum_{i\not=j}e^{-\frac{\eta|y_i-y_j|}{\varepsilon}}\Big) \\
  &\ +\varepsilon^{3}O\Big(\sum_{i=1}^k|Q(y_i)-1|^{2(1-\tau)}+\sum_{i=1}^k\sum_{m=1}^{[h]}\varepsilon^{2(m-\tau)}|D^mQ(y_i)|^{2(1-\tau)}\Big)\\   &\
+\varepsilon^{3}O\Big(\varepsilon^{2([h]+1-\tau)}+\sum_{i\not=j}e^{-2(\bar{\theta}-\tau)\frac{\eta|y_i-y_j|}{\varepsilon}}\Big).
   \end{aligned}
   \end{equation}
Choose\ $\tau$ sufficiently small such that
\begin{equation*}
  2([h]+1-\tau )>{[h]+1},\ 2(m-\tau)>m,\ 2(\bar{\theta}-\tau)>1.
\end{equation*}
Then by\ (\ref{44}), we have
\begin{equation}\label{45}
\begin{aligned}
   K(Y) =&\ \varepsilon^3\left(kA-k^2B\right)-\frac{2}{q-2}A\varepsilon^3 \sum_{i=1}^k \left(Q(y_i)-1\right)\\   &\
   -\int \sum_{i=1}^{k-1}w_{\varepsilon,y_i}\Big(\sum_{j=i+1}^kw_{\varepsilon,y_j}\Big)^{q-1}+O\big(\varepsilon^{4+[h]}\big)\\   &\ +O\Big(\sum_{i=1}^k\sum_{m=1}^{[h]}\varepsilon^{3+m}|D^mQ(y_i)|+\varepsilon^3 \sum_{i\not=j}e^{-\frac{\eta|y_i-y_j|}{\varepsilon}}\Big).
   \end{aligned}
   \end{equation}
Combining the above equality and the condition $(Q_3)$ yields
\begin{eqnarray}\label{46}
   K(Y) \le \varepsilon^3\Big(kA-k^2B-C\sum_{i=1}^k P_1(y'_i)-C\sum_{i<j}e^{-\frac{\eta|y_i-y_j|}{\varepsilon}}\Big)+O(\varepsilon^{4}).
   \end{eqnarray}
If\ $\frac{{\eta |{y_i} - {y_j}|}}{\varepsilon }={R_1}$ for some $i\ne j$, taking\ $R_1=\frac {\alpha h_1 ln{\frac {1}{\varepsilon}}-lnT} {2}$, by\ (\ref{46}) we obtain
\begin{equation}\label{47}
   K(Y) \le \varepsilon^3\left(kA-k^2B-T\varepsilon^{\alpha h_1}\right)-C\varepsilon^3\sum_{i=1}^k P_1(y'_i)+O(\varepsilon^{4})
   <c_{\varepsilon,1}.
   \end{equation}
If\ ${y_i} \in \partial B_\delta ^t(0) \times B_{\varepsilon^\alpha} ^{3 - t}(0)$ for some\ $i\in \{1,\ldots,k\}$, combining\ (\ref{46}) and\ (\ref{112}) yields
\begin{equation}\label{48}
\begin{aligned}
   K(Y) &\le \varepsilon^3\Big(kA-k^2B-C\lambda \sum_{i=1}^k|y'_i|^{h_1}\Big)-C\varepsilon^3\sum_{i<j}e^{-\frac{\eta|y_i-y_j|}{\varepsilon}}+O(\varepsilon^{4})\nonumber \\&\le  \varepsilon^3\left(kA-k^2B-C\lambda \varepsilon^{\alpha h_1}\right)-C\varepsilon^3\sum_{i<j}e^{-\frac{\eta|y_i-y_j|}{\varepsilon}}+O(\varepsilon^{4}).
   \end{aligned}
   \end{equation}
Let\ $T$ sufficiently small such that\ $T<C\lambda $, then we have\ $K(Y)<c_{\varepsilon,1}$.

Step 2: Suppose\ $t\in\{1,2\}$ and\ ${y_j} \in  B_\delta ^t(0) \times \partial B_{\varepsilon^\alpha} ^{3 - t}(0)$ for some\ $j\in \{1,\ldots,k\}$. Without loss of generality, we assume\ $j=1$. We claim that either\ $K(Y)<c_{\varepsilon,1}$ or\ $\frac {\partial K(Y)} {\partial n}>0$, where\ $n$ is  the outward unit normal of\ $B_\delta ^t(0) \times \partial B_{\varepsilon^\alpha} ^{3 - t}(0)$ at\ $y_1$.

In fact, for any\ ${y_i} \in  B_\delta ^t(0) \times B_{\varepsilon^\alpha} ^{3 - t}(0)$, and $m\ge 1$, we have
\begin{equation}\label{49}
\begin{aligned}
  \varepsilon^m|D^mQ(y_i)|&= O\left( \varepsilon^m|y'_i|^{h_1-m}+\varepsilon^m|y''_i|^{h-m}\right) \\
  &= O\left( \varepsilon^m|y'_i|^{h_1-m}+\varepsilon^{\alpha h+m(1-\alpha)}\right).
   \end{aligned}
   \end{equation}
By Lemma \ref{lemma A.2}, we obtain
\begin{equation}\label{410}
\begin{aligned}
  \frac {\partial K} {\partial y_{1l}}=&\  -C\varepsilon^3D_lQ(y_1)-(q-1)\sum_{i=2}^k\int w_{\varepsilon,y_1}^{q-2} w_{\varepsilon,y_i}\frac {\partial w_{\varepsilon,y_1}} {\partial y_{1l}} \\
  &\ +O\Big(\sum_{i=1}^k \sum_{m=2}^{[h]}\varepsilon^{2+m-\tau}|D^mQ(y_i)|^{1-\tau}+\varepsilon^{3+[h]-\tau}\Big) \\
  &\ +O\Big(\varepsilon^2\sum_{i=1}^k|Q(y_i)-1|^{1-\tau}\Big)+O\Big(\varepsilon^2\sum_{i\not =j}e^{-(\bar\theta-\tau)\frac{\eta|y_i-y_j|}{\varepsilon}}\Big).
   \end{aligned}
   \end{equation}
Denote\ $\bar\eta=\min_{i\not= j}\eta|y_i-y_j|$. We divide it into two cases.

\ $(i)$ Suppose that\ $e^{-\frac {\bar\eta} {\varepsilon}}>L\varepsilon^{\alpha h}$ or\ $|y'_i|>L\varepsilon^{{\alpha h}/h_1}$ for some\ $i\in \{1,\ldots,k\}$, where\ $L>T$ is a large constant. We claim that\ $K(Y)<c_{\varepsilon,1}$.

In fact, combining\ (\ref{45}) and\ (\ref{112}) yields
\begin{equation}\label{411}
\begin{aligned}
   K(Y) &\le \varepsilon^3\left(kA-k^2B\right)-C_1 \varepsilon^3 \sum_{i=1}^k |y'_i|^{h_1}-C_1 \varepsilon^3 e^{-\frac {\bar\eta} {\varepsilon}}+O\Big(\sum_{i=1}^k \sum_{m=1}^{[h]}\varepsilon^{3+m}|y'_i|^{h_1-m}+\varepsilon^{3+\alpha h}\Big)
   \\ &\le \varepsilon^3\left(kA-k^2B\right)-(C_1-\tau' )\varepsilon^3\sum_{i=1}^k |y'_i|^{h_1}-C_1\varepsilon^3 e^{-\frac {\bar\eta} {\varepsilon}}+C_{\tau'} \varepsilon^{3+\alpha h},
   \end{aligned}
   \end{equation}
where\ $\tau'>0$ is a constant. When\ $L>T$ is large enough, we have\ $K(Y)<c_{\varepsilon,1}$.

\ $(ii)$ Suppose that\ $e^{-\frac {\bar\eta} {\varepsilon}}\le L\varepsilon^{\alpha h}$ and\ $|y'_i|\le L\varepsilon^{{\alpha h}/h_1}$,\ $i=1,\ldots,k$. We claim that\ $\frac {\partial K(Y)} {\partial n} >0$.\\
First, we can see
\begin{equation}\label{412}
|1-Q(y_i)|=O(\varepsilon^{\alpha h}),
\end{equation}
\begin{equation}\label{413}
|D^mQ(y_i)|\varepsilon^m=O(\varepsilon^{\alpha h(h_1-m)/h_1}\varepsilon^m+\varepsilon^{\alpha h+m(1-\alpha)})=O(\varepsilon^{\alpha h+m(1-\alpha)}),
\end{equation}
\begin{equation}\label{414}
\varepsilon^2 e^{-\bar\theta \frac {\bar\eta} {\varepsilon}}=O(\varepsilon^{2+\bar\theta \alpha h}).
\end{equation}
Since for any\ $i\not =1$,
\begin{equation}\label{415}
\int w_{\varepsilon,y_1}^{q-2} w_{\varepsilon,y_i}\frac {\partial w_{\varepsilon,y_1}} {\partial y_{1l}}=\big(C+o(1)\big)\varepsilon^2 w\Big(\frac {|y_i-y_1|}{\varepsilon}\Big) \frac{y_{il}-y_{1l}}{|y_i-y_1|} ,
\end{equation}
\begin{equation}\label{416}
\bigg\langle {\frac{y_{i}-y_{1}}{|y_i-y_1|}},n \bigg\rangle \le 0,\forall \ y_i\in B_\delta ^t(0) \times B_{\varepsilon^\alpha} ^{3 - t}(0),
\end{equation}
where
\begin{equation*}
n=\left\{
\begin{aligned}
&\bigg(0,\frac{y_{1,2}}{(y_{1,2}^2+y_{1,3}^2)^{\frac {1}{2}}},\frac{y_{1,3}}{(y_{1,2}^2+y_{1,3}^2)^{\frac {1}{2}}}\bigg), & t=1, \\
&\Big(0,0,\frac {y_{1,3}}{|y_{1,3}|}\Big), & t=2,
\end{aligned}
\right.
\end{equation*}
combining\ (\ref{114}), (\ref{410}) and (\ref{412})--(\ref{416}) yields
\begin{equation}\label{417}
\begin{aligned}
  \frac {\partial {K(Y)}} {\partial {n}} &\ge C\varepsilon^3\langle {-DQ(y_1)},{n} \rangle+O(\varepsilon^{3+\alpha h+\tau''})\\
  &\ge C\varepsilon^3|y''_1|^{h-1}+O(\varepsilon^{3+\alpha h+\tau''})  \\
   &>0,
   \end{aligned}
   \end{equation}
where\ $\tau''>0$ is a constant.

Combining Steps 1 and 2 we complete the proof of this lemma.
\end{proof}

We are now ready to prove Theorem\ \ref{thm1.1}.

\begin{proof}[\textbf{Proof of Theorem\ \ref{thm1.1}:}]
 Combining Lemma\ \ref{lm2.6} and\ \ref{lm4.1}, we just need to prove\ $K^{c_{\varepsilon,2}}$ cannot be deformed into\ $K^{c_{\varepsilon,1}}$. Then\ $K$ has at least one critical point in \ $K^{c_{\varepsilon,2}}\backslash K^{c_{\varepsilon,1}}$. Finally, by Lemma\ \ref{lm3.7}, we obtain that\ $u=\sum_{i=1}^k w_{\varepsilon,y_{\varepsilon,i}}+\varphi_{\varepsilon,Y}$ is a critical point of\ $I_{\varepsilon}$, in the sense that it is a solution of equation\ (\ref{11}).

Next, we prove\ $K^{c_{\varepsilon,2}}$ cannot be deformed into\ $K^{c_{\varepsilon,1}}$. It's easy to know
\begin{equation*}
K^{c_{\varepsilon,2}}=\Omega_{\varepsilon^\alpha}.
\end{equation*}
Denote
\begin{equation}\label{421}
M=B_\delta ^t(0) \times B_{\varepsilon^\alpha} ^{3 - t}(0),
\end{equation}
\begin{equation}\label{422}
\Gamma_\iota=\big\{(y',y'')\in M,\ |y'|\ge \iota \big\},
\end{equation}
\begin{equation}\label{423}
T_\gamma=\cup_{i\not =j}\big\{\eta|y_i-y_j|\le \gamma,\ y_i,\ y_j\in M\big\},
\end{equation}
\begin{equation}\label{424}
L_{\iota,\gamma}=(\underbrace {\Gamma_\iota \times M \times \cdots \times M}_k) \cup (\underbrace {M \times \Gamma_\iota \times \cdots \times M}_k) \cup \cdots \cup (\underbrace {M \times \cdots \times M \times \Gamma_\iota}_k ) \cup T_\gamma.
\end{equation}

Step 1: We claim that there exist constants\ $C$,\ $c'$ with\ $C>c'>0$, such that
\begin{equation}\label{425}
{{L_{{\delta/2},{c'\varepsilon ln\varepsilon^{-1}}}\setminus T_{\varepsilon R_1}}} \subset K^{c_{\varepsilon,1}} \subset {{{L_{{c'\varepsilon^{\alpha h/h_1}},{C\varepsilon ln\varepsilon^{-1}}}}  \backslash T_{\varepsilon R_1}}}.
\end{equation}
In fact, for any\ $Y\in K^{c_{\varepsilon,1}}$, we have $K(Y)<c_{\varepsilon,1}$. Then by\ (\ref{45}), we obtain
\begin{equation}\label{426}
\begin{aligned}
    c_{\varepsilon,1}=&\ \varepsilon^3(kA-k^2B-T\varepsilon^{\alpha h_1}) \\
    >&\ K(Y)  \\
     \ge&\ \varepsilon^3\left(kA-k^2B\right)-c'\varepsilon^3\sum_{i=1}^k |y'_i|^{h_1}-c'\varepsilon^3\sum_{i\not =j}e^{-\frac{\eta|y_i-y_j|}{\varepsilon}}+O\left(\varepsilon^{4+\alpha h}\right).
   \end{aligned}
   \end{equation}
Thus,\ $|y'_i|\ge c'\varepsilon^{\alpha h/h_1}$ or\ $\eta|y_i-y_j|\le C\varepsilon ln\varepsilon^{-1}$ for some\ $i\not =j$. Hence,
\begin{equation*}
K^{c_{\varepsilon,1}} \subset {{{L_{{c'\varepsilon^{\alpha h/h_1}},{C\varepsilon ln\varepsilon^{-1}}}}  \backslash T_{\varepsilon R_1}}}.
\end{equation*}
On the other hand, choose\ $c'>0$ sufficiently small. When\ $|y'_i|\ge \frac{\delta}{2}$ or\ $\eta|y_i-y_j|\le c'\varepsilon ln\varepsilon^{-1}$ for some\ $i\not =j$, by\ (\ref{46}), we have\ $K(Y)<c_{\varepsilon,1}$. Then
\begin{equation*}
{{L_{{\delta/2},{c'\varepsilon ln\varepsilon^{-1}}}\setminus T_{\varepsilon R_1}}} \subset K^{c_{\varepsilon,1}}.
\end{equation*}
So the claim follows.

Step 2: Since\ $  {{{L_{{c'\varepsilon^{\alpha h/h_1}},{C\varepsilon ln\varepsilon^{-1}}}}  \backslash T_{\varepsilon R_1}}}$ can be deformed into\ ${{L_{{\delta/2},{c'\varepsilon ln\varepsilon^{-1}}}\setminus T_{\varepsilon R_1}}}$, then\ $K^{c_{\varepsilon,1}}$ can be deformed into\ ${{L_{{\delta/2},{c'\varepsilon ln\varepsilon^{-1}}}\setminus T_{\varepsilon R_1}}}$. Suppose\ $K^{c_{\varepsilon,2}}$ can be deformed into\ $K^{c_{\varepsilon,1}}$, then we see that\ $\Omega_{\varepsilon^\alpha}=K^{c_{\varepsilon,2}}$ can be deformed into\ ${{L_{{\delta/2},{c'\varepsilon ln\varepsilon^{-1}}}\setminus T_{\varepsilon R_1}}}$. Hence, $\underbrace {M \times M\times \cdots \times M}_k$ can be deformed into\ $L_{{\delta/2},{c'\varepsilon ln\varepsilon^{-1}}}$. However,
\begin{equation*}
H^t(M,\Gamma_{\delta/2})=H^t(B_{\delta}^t(0),\partial B_{\delta}^t(0))\not =0,
\end{equation*}
By Lemma\ \ref{lm4.2}, we obtain
\begin{equation*}
H_*(\underbrace {M \times M\times \cdots \times M}_k,L_{{\delta/2},{c'\varepsilon ln\varepsilon^{-1}}})\not =0,
\end{equation*}
Then\ $\underbrace {M \times M\times \cdots \times M}_k$ cannot be deformed into\ $L_{{\delta/2},{c'\varepsilon ln\varepsilon^{-1}}}$. This is a contradiction.
\end{proof}

\section*{Appendix}

\appendix

\section{Energy estimates }

\renewcommand{\theequation}{A.\arabic{equation}}

\begin{lem}\label{lemma A.1}
For\ $\varepsilon$ sufficiently small and any\ $Y\in D^k_{\varepsilon,\delta}$, we have
\begin{equation}\label{325}
\begin{aligned}
  J_{\varepsilon}(Y,0)=&\ \varepsilon^3\Big(kA-k^2B-\frac{2}{q-2}A \sum_{i=1}^k \left(Q(y_i)-1\right)\Big)\\   &\
  -\int \sum_{i=1}^{k-1}w_{\varepsilon,y_i}\Big(\sum_{j=i+1}^kw_{\varepsilon,y_j}\Big)^{q-1}+O\Big(\sum_{i=1}^k\sum_{m=1}^{[h]}\varepsilon^{3+m}|D^mQ(y_i)|\Big)\\   &\
  +O\Big(\varepsilon^3\sum_{i\not=j}e^{-\frac{\eta|y_i-y_j|}{\varepsilon}}+\varepsilon^{4+[h]}\Big),
\end{aligned}
\end{equation}
where\ $A=\frac{q-2}{2q}\|w\|_{L^q}^q,\ B=\frac{b}{4}\|\nabla w\|_{L^2}^4$.
\end{lem}
\begin{proof}
By the definition of\ $J_{\varepsilon}(Y,\varphi)$, we obtain
\begin{equation}\label{326}
\begin{aligned}
 J_{\varepsilon}(Y,0)=&\ I_{\varepsilon}(W_{\varepsilon,Y})\\
=&\ \frac 12 \|W_{\varepsilon,Y}\|_{\varepsilon}^2+\frac {\varepsilon b}{4} \bigg(\int|\nabla W_{\varepsilon,Y} |^2  \bigg)^2-\frac 1q \int Q(x)W_{\varepsilon,Y}^q \\
=&\
\frac 12 \sum_{i=1}^k\int\left(\varepsilon^2a|\nabla w_{\varepsilon,y_i}|^2+w_{\varepsilon,y_i}^2\right)+\frac 12 \sum_{i\not=j}\int\left(\varepsilon^2a|\nabla w_{\varepsilon,y_i}\nabla w_{\varepsilon,y_j}|+w_{\varepsilon,y_i}w_{\varepsilon,y_j}\right)\\
&\ +\frac {\varepsilon b}{4} \bigg(\sum_{i=1}^k\int|\nabla w_{\varepsilon,y_i} |^2  \bigg)^2+\frac {\varepsilon b}{4} \bigg(\sum_{i\not=j}\int|\nabla w_{\varepsilon,y_i} \nabla w_{\varepsilon,y_j} |  \bigg)^2-\frac 1q \int Q(x)W_{\varepsilon,Y}^q\\
=&\ \frac 12 \sum_{i=1}^k\int\left(\varepsilon^2a|\nabla w_{\varepsilon,y_i}|^2+w_{\varepsilon,y_i}^2\right)+\frac {\varepsilon b}{4} \bigg(\sum_{i=1}^k\int|\nabla w_{\varepsilon,y_i} |^2  \bigg)^2\\
&\ -\frac 1q \int Q(x)W_{\varepsilon,Y}^q+O\Big(\varepsilon^3\sum_{i\not=j}e^{-\frac{\eta|y_i-y_j|}{\varepsilon}}\Big).
\end{aligned}
\end{equation}
As\ $Q(0)=1$ and\ $w$ is the solution of\ (\ref{25}), we obtain that\ $w_{\varepsilon,y_i}(1\le i\le k)$ satisfies
\begin{equation*}
  -\left(\varepsilon^2a+\varepsilon bk\int|\nabla w_{\varepsilon,y_j}|^2\right)\Delta w_{\varepsilon,y_i}+w_{\varepsilon,y_i}=w_{\varepsilon,y_i}^{q-1},\ j=1,\ldots,k.
\end{equation*}
Multiplying\ $w_{\varepsilon,y_i}$ on both sides of the above equality and integrating, we have
\begin{equation*}
 \left(\varepsilon^2a+\varepsilon bk\int|\nabla w_{\varepsilon,y_j}|^2\right)\int|\nabla w_{\varepsilon,y_i}|^2+\int w_{\varepsilon,y_i}^2
=\int w_{\varepsilon,y_i}^{q}.
  \end{equation*}
Sum $i$ from $1$ to $k$, we obtain
\begin{equation*}
  \sum_{i=1}^k\left(\varepsilon^2a+\varepsilon bk\int|\nabla w_{\varepsilon,y_j}|^2\right)\int|\nabla w_{\varepsilon,y_i}|^2+\sum_{i=1}^k\int w_{\varepsilon,y_i}^2
=\int \sum_{i=1}^kw_{\varepsilon,y_i}^{q}.
\end{equation*}
Then
\begin{eqnarray*}
  \sum_{i=1}^k\int\left(\varepsilon^2 a|\nabla w_{\varepsilon,y_i}|^2 +w_{\varepsilon,y_i}^2\right)= -\varepsilon bk\int|\nabla w_{\varepsilon,y_j}|^2\int\sum_{i=1}^k|\nabla w_{\varepsilon,y_i}|^2+\int \sum_{i=1}^kw_{\varepsilon,y_i}^{q}.
\end{eqnarray*}
Substituting it into\ (\ref{326}) yields
\begin{equation}\label{327}
\begin{aligned}
 J_{\varepsilon}(Y,0)=&\ \frac 12  \sum_{i=1}^k\int w_{\varepsilon,y_i}^q-\frac {\varepsilon b}{4} \bigg(\sum_{i=1}^k\int|\nabla w_{\varepsilon,y_i} |^2  \bigg)^2-\frac 1q \int Q(x)W_{\varepsilon,Y}^q+O\Big(\varepsilon^3\sum_{i\not=j}e^{-\frac{\eta|y_i-y_j|}{\varepsilon}}\Big) \\
=&\ \Big(\frac 12-\frac 1q \Big) \sum_{i=1}^k\int w_{\varepsilon,y_i}^q-\frac {\varepsilon b}{4} \bigg(\sum_{i=1}^k\int|\nabla w_{\varepsilon,y_i} |^2  \bigg)^2 \\
&\ -\frac 1q \int \bigg(Q(x)\Big(\sum_{i=1}^kw_{\varepsilon,y_i} \Big)^q-\sum_{i=1}^kw_{\varepsilon,y_i}^q \bigg)+O\Big(\varepsilon^3\sum_{i\not=j}e^{-\frac{\eta|y_i-y_j|}{\varepsilon}}\Big)\\
=&\ \Big(\frac 12-\frac 1q \Big) \varepsilon^3k \|w\|_{L^q}^q-\frac 14 \varepsilon^3k^2b\|\nabla w\|_{L^2}^4
-\frac 1q  \int \bigg(\Big(\sum_{i=1}^kw_{\varepsilon,y_i}\Big)^q-\sum_{i=1}^kw_{\varepsilon,y_i}^q \bigg) \\
&\ -\frac 1q \int \left(Q(x)-1\right)\Big(\sum_{i=1}^kw_{\varepsilon,y_i}\Big)^q+O\Big(\varepsilon^3\sum_{i\not=j}e^{-\frac{\eta|y_i-y_j|}{\varepsilon}}\Big).
\end{aligned}
\end{equation}
Next, we estimate the third and fourth term of the right side of \eqref{327} respectively. First, to estimate the third term, we use the following inequalities
\begin{equation*}
\begin{aligned}
  \left||a+b|^q-a^q-b^q-qa^{q-1}b-qab^{q-1} \right| &\le \left\{
                                                 \begin{array}{ll}
                                                 C|b|^{q-2}|a|^2, & \hbox{$|b|\le|a|$;} \\
                                                 C|a|^{q-2}|b|^2, & \hbox{$|a|\le|b|$,}
                                                 \end{array}
                                               \right.
  \\
  &\le C|a|^{\frac q2}|b|^{\frac q2},\quad(2<q\le3),\\
   \left||a+b|^q-a^q-b^q-qa^{q-1}b-qab^{q-1} \right| &\le C(a^{q-2}b^2+a^2b^{q-2}),\quad(q>3).
\end{aligned}
\end{equation*}
Thus,
\begin{equation}\label{329}
\begin{aligned}
 &\ \Big(\sum_{i=1}^k w_{\varepsilon,y_i}\Big)^q - w_{\varepsilon,y_1}^q-\Big(\sum_{i=2}^k w_{\varepsilon,y_i}\Big)^q - q w_{\varepsilon,y_1}^{q-1}\Big(\sum_{i=2}^k w_{\varepsilon,y_i}\Big)
 -qw_{\varepsilon,y_1}\Big(\sum_{i=2}^k w_{\varepsilon,y_i}\Big)^{q-1}\\
 \le&\ \left\{
            \begin{array}{ll}
              C w_{\varepsilon,y_1}^{\frac{q}{2}}\Big(\sum\limits_{i=2}^k w_{\varepsilon,y_i}\Big)^{\frac{q}{2}}, & \hbox{$2<q\le3$,} \\
              Cw_{\varepsilon,y_1}^{q-2}\Big(\sum\limits_{i=2}^k w_{\varepsilon,y_i}\Big)^2+Cw_{\varepsilon,y_1}^2\Big(\sum\limits_{i=2}^k w_{\varepsilon,y_i}\Big)^{q-2}, & \hbox{$q>3$.}
            \end{array}
          \right.
\end{aligned}
\end{equation}
If\ $2<q\le3$, by\ \eqref{216}, we have
\begin{equation}
\int w_{\varepsilon,y_1}^{\frac{q}{2}}\Big(\sum_{i=2}^k w_{\varepsilon,y_i}\Big)^{\frac{q}{2}} \le C\int\sum_{i=2}^k w_{\varepsilon,y_1}^{\frac{q}{2}}w_{\varepsilon,y_i}^{\frac{q}{2}}
   \le C\sum_{i=2}^k \varepsilon^3 e^{-\frac q2\frac{\eta|y_1-y_i|}{\varepsilon}}.
\end{equation}
If\ $q>3$, by\ \eqref{216}, we have
\begin{equation}\label{331}
  \int \bigg(  w_{\varepsilon,y_1}^{q-2}\Big(\sum\limits_{i=2}^k w_{\varepsilon,y_i}\Big)^2+w_{\varepsilon,y_1}^2\Big(\sum\limits_{i=2}^k w_{\varepsilon,y_i}\Big)^{q-2}    \bigg)
  \le\left\{
                                               \begin{array}{ll}
                                                 C\varepsilon^3\sum\limits_{i=2}^k e^{-(q-2)\frac{\eta|y_1-y_i|}{\varepsilon}}, & \hbox{$3<q\le4$;} \\
                                                C\varepsilon^3\sum\limits_{i=2}^k e^{-2\frac{\eta|y_1-y_i|}{\varepsilon}}, & \hbox{$q>4$.}
                                               \end{array}
                                             \right.
\end{equation}
Denote
\begin{equation*}
  1+\bar\sigma=\left\{
                 \begin{array}{ll}
                   \frac q2, & \hbox{$2<q\le3$,} \\
                   q-2, & \hbox{$3<q\le4$,} \\
                   2, & \hbox{$q>4$.}
                 \end{array}
               \right.
\end{equation*}
Combining\ (\ref{329})--(\ref{331}) yields
\begin{equation}\label{332}
\begin{aligned}
   \int \bigg(\Big(\sum_{i=1}^kw_{\varepsilon,y_i}\Big)^q-\sum_{i=1}^kw_{\varepsilon,y_i}^q \bigg) =&\  q\int \sum_{i<j}w_{\varepsilon,y_i}^{q-1} w_{\varepsilon,y_j}+q\int \sum_{i=1}^{k-1} w_{\varepsilon,y_i}\Big(\sum_{j=i+1}^{k}w_{\varepsilon,y_j}\Big)^{q-1}\\
  &\ +O\Big(\varepsilon^3\sum_{i<j}e^{-(1+\bar\sigma)\frac{\eta|y_i-y_j|}{\varepsilon}}\Big)\\
=&\ q\int \sum_{i=1}^{k-1} w_{\varepsilon,y_i}\Big(\sum_{j=i+1}^{k}w_{\varepsilon,y_j}\Big)^{q-1}\\
  &\
+O\Big(\varepsilon^3\sum_{i<j}e^{-\frac{\eta|y_i-y_j|}{\varepsilon}}\Big).
\end{aligned}
\end{equation}
Secondly, to estimate the fourth term of \eqref{327}, we have
\begin{eqnarray}\label{333}
 \int \left(Q(x)-1\right)\Big(\sum_{i=1}^kw_{\varepsilon,y_i}\Big)^q= \sum_{i=1}^k\int \left(Q(x)-1\right)w_{\varepsilon,y_i}^q
 +O\Big(\varepsilon^3\sum_{i\not=j}e^{-\frac{\eta|y_i-y_j|}{\varepsilon}}\Big).
\end{eqnarray}
Estimating the first term of the right side of\ (\ref{333}), we obtain
\begin{equation}\label{334}
\begin{aligned}
 \int_{\mathbb R^3} \left(Q(x)-1\right) w_{\varepsilon,y_i}^q
   =&\ \int_{B_\delta(y_i)}\left(Q(x)-Q(y_i)\right)w_{\varepsilon,y_i}^q+\int_{B_\delta^c(y_i)}\left(Q(x)-Q(y_i)\right) w_{\varepsilon,y_i}^q \\
   &\ +\int_{\mathbb R^3}  \left(Q(y_i)-1\right)w_{\varepsilon,y_i}^q,
\end{aligned}
\end{equation}
\begin{equation}\label{335}
\begin{aligned}
  \Big|\int_{B_\delta^c(y_i)}\left(Q(x)-Q(y_i)\right) w_{\varepsilon,y_i}^q \Big|&\le\int_{B_\delta^c(y_i)}|Q(x)-Q(y_i)|w_{\varepsilon,y_i}^q
  \le  C\varepsilon^3\int _{|y|\ge \frac \delta\varepsilon}w^q(y) \\
   &\le C \varepsilon^3 \int _{|y|\ge \frac \delta\varepsilon}e^{- q\eta |y|}|y|^{-q}
   \le C\varepsilon^3 e^{-\frac{\bar q\eta \delta}{\varepsilon}},
\end{aligned}
\end{equation}
where\ $\bar q=q-\hat{\theta}$ with $\hat{\theta}>0$ is a small constant, and
\begin{equation}\label{336}
\begin{aligned}
   \int_{B_\delta(y_i)} \left(Q(x)-Q(y_i)\right)w_{\varepsilon,y_i}^q &\le C\varepsilon^3\int_{|y|< \frac{\delta}{\varepsilon}}\Big(\sum_{m=1}^{[h]} \varepsilon^m|y|^m|D^mQ(y_i)|+\varepsilon^{[h]+1}|y|^{[h]+1}\Big) w^q(y) \\
   &\le C\Big(\sum_{m=1}^{[h]} \varepsilon^{3+m}|D^mQ(y_i)|+\varepsilon^{4+[h]}\Big).
\end{aligned}
\end{equation}
Combining\ (\ref{334})--(\ref{336}) yields
\begin{equation}\label{337}
\begin{aligned}
   \int \left(Q(x)-1\right)w_{\varepsilon,y_i}^q =&\ \left(Q(y_i)-1\right)\int w_{\varepsilon,y_i}^q+\varepsilon^3O\big(e^{-\frac{\bar q\eta \delta}{\varepsilon}}+\varepsilon^{[h]+1}\big) +\varepsilon^3 O\Big(\sum_{m=1}^{[h]} \varepsilon^m|D^mQ(y_i)|\Big) \\
   =&\ \left(Q(y_i)-1\right)\varepsilon^3\|w\|_{L^q}^q+O\big(\varepsilon^{4+[h]}\big)
   +O\Big(\sum_{m=1}^{[h]} \varepsilon^{3+m}|D^mQ(y_i)|\Big).
\end{aligned}
\end{equation}
Combining\ (\ref{327}), (\ref{332}), (\ref{333}) and (\ref{337}) yields\ (\ref{325}).
\end{proof}
\begin{lem}\label{lemma A.2}
For any\ $Y\in D^k_{\varepsilon,\delta}$, there holds
\begin{equation}\label{517}
\begin{aligned}
  \frac {\partial K} {\partial y_{il}} =&\ -C\varepsilon^3D_lQ(y_i)-(q-1)\sum_{j=1,j\not =i}^k\int w_{\varepsilon,y_i}^{q-2} w_{\varepsilon,y_j}\frac {\partial w_{\varepsilon,y_{i}}} {\partial y_{il}} \\ &\  +O\Big(\sum_{i=1}^k \sum_{m=2}^{[h]}\varepsilon^{2+m-\tau}|D^mQ(y_i)|^{1-\tau}+\varepsilon^{3+[h]-\tau}\Big) \\
  &\ +O\Big(\varepsilon^2\sum_{i=1}^k|Q(y_i)-1|^{1-\tau}\Big)+O\Big(\varepsilon^2\sum_{i\not =j}e^{-(\bar\theta-\tau)\frac{\eta|y_i-y_j|}{\varepsilon}}\Big),
   \end{aligned}
   \end{equation}
where\ $i=1,\ldots,k$ and $l=1,2,3$.
\end{lem}
\begin{proof}
By the definition of $K(Y)$, we have
\begin{equation}\label{518}
 \frac {\partial K} {\partial y_{il}}= \frac {\partial J_\varepsilon} {\partial y_{il}}+\bigg\langle {\frac {\partial J_\varepsilon} {\partial \varphi_{\varepsilon,Y}}},{\frac {\partial \varphi_{\varepsilon,Y}} {\partial y_{il}}} \bigg\rangle.
\end{equation}
First, estimating the first term of \eqref{518}, we obtain
\begin{equation}\label{519}
\begin{aligned}
  \frac {\partial J_\varepsilon} {\partial y_{il}}=&\ \int \varepsilon^2a\nabla ({W_{\varepsilon ,Y}}+\varphi_{\varepsilon,Y})\nabla {\frac{\partial {W_{\varepsilon,Y}}}{\partial y_{il}}}+\int {({W_{\varepsilon ,Y}}+\varphi_{\varepsilon,Y})\frac{\partial {W_{\varepsilon,Y}}}{\partial y_{il}}} \\
  &\ +
  \varepsilon b{\int {\left| {\nabla ({W_{\varepsilon ,Y}+\varphi_{\varepsilon,Y}})} \right|} ^2}\int {\nabla ({W_{\varepsilon ,Y}+\varphi_{\varepsilon,Y}} )}\nabla {\frac{\partial {W_{\varepsilon,Y}}}{\partial y_{il}}}
  -\int {Q(x)({W_{\varepsilon ,Y}}+\varphi_{\varepsilon,Y})^{q-1}_+{\frac{\partial {W_{\varepsilon,Y}}}{\partial y_{il}}}}  \\
  =&\ \bigg\langle {{I_\varepsilon }'\left( {{W_{\varepsilon ,Y}}} \right), \frac{\partial {W_{\varepsilon,Y}}}{\partial y_{il}}} \bigg\rangle+ \bigg\langle {{I_\varepsilon }''\left( {{W_{\varepsilon ,Y}}} \right)\left[ \varphi_{\varepsilon,Y} \right],\frac{\partial {W_{\varepsilon,Y}}}{\partial y_{il}} } \bigg\rangle+\bigg\langle {{R'_{\varepsilon ,Y}}\left( \varphi_{\varepsilon,Y}  \right)},{\frac{\partial {W_{\varepsilon,Y}}}{\partial y_{il}}} \bigg\rangle  \\
  =&\ {l_{\varepsilon ,Y}} \bigg(\frac{\partial {W_{\varepsilon,Y}}}{\partial y_{il}} \bigg) +\bigg\langle {{I_\varepsilon }''\left( {{W_{\varepsilon ,Y}}} \right)\left[ \varphi_{\varepsilon,Y} \right],\frac{\partial {W_{\varepsilon,Y}}}{\partial y_{il}} } \bigg\rangle + \bigg\langle {{R'_{\varepsilon ,Y}}( \varphi_{\varepsilon,Y}  )},{\frac{\partial {W_{\varepsilon,Y}}}{\partial y_{il}}} \bigg\rangle.
   \end{aligned}
   \end{equation}
To estimate the first term of the right side of \eqref{519}, since\ $Q(0)=1$ and\ $w$ is the solution of\ (\ref{25}), we obtain\ $w_{\varepsilon,y_j}(1\le j\le k)$ satisfies
\begin{equation*}
  -\bigg(\varepsilon^2a+\varepsilon bk\int|\nabla w_{\varepsilon,y_t}|^2\bigg)\Delta w_{\varepsilon,y_j}+w_{\varepsilon,y_j}=w_{\varepsilon,y_j}^{q-1},\ t=1,\ldots,k.
\end{equation*}
Sum $j$ from $1$ to $k$, we obtain
\begin{equation*}
  -\bigg(\varepsilon^2a+\varepsilon bk\int|\nabla w_{\varepsilon,y_t}|^2\bigg)\Delta W_{\varepsilon,Y}+W_{\varepsilon,Y}=\sum_{j=1}^kw_{\varepsilon,y_j}^{q-1}.
  \end{equation*}
Multiplying $\frac{\partial {W_{\varepsilon,Y}}}{\partial y_{il}}$ on both sides of the above equality and integrating, we have
\begin{equation*}
  \bigg(\varepsilon^2a+\varepsilon bk\int|\nabla w_{\varepsilon,y_t}|^2\bigg)\int\nabla W_{\varepsilon,Y}\nabla{\frac{\partial {W_{\varepsilon,Y}}}{\partial y_{il}}}+\int W_{\varepsilon,Y}\frac{\partial {W_{\varepsilon,Y}}}{\partial y_{il}}
=\int \sum_{j=1}^kw_{\varepsilon,y_j}^{q-1}\frac{\partial {W_{\varepsilon,Y}}}{\partial y_{il}}.
\end{equation*}
Then
\begin{equation*}\label{520}
\begin{aligned}
  \bigg\langle  W_{\varepsilon,Y},{\frac{\partial {W_{\varepsilon,Y}}}{\partial y_{il}}}\bigg\rangle_\varepsilon  &= \int\Big(\varepsilon^2 a\nabla W_{\varepsilon,Y}\nabla{\frac{\partial {W_{\varepsilon,Y}}}{\partial y_{il}}} +W_{\varepsilon,Y}\frac{\partial {W_{\varepsilon,Y}}}{\partial y_{il}}\Big)\\
   &= -\varepsilon bk\int|\nabla w_{\varepsilon,y_t}|^2\int \nabla W_{\varepsilon,Y}\nabla\frac{\partial {W_{\varepsilon,Y}}}{\partial y_{il}}+\int \sum_{j=1}^kw_{\varepsilon,y_j}^{q-1}\frac{\partial {W_{\varepsilon,Y}}}{\partial y_{il}},
\end{aligned}
\end{equation*}
Hence,
\begin{equation}\label{570}
\begin{aligned}
  {l_{\varepsilon ,Y}}\bigg( \frac{\partial {W_{\varepsilon,Y}}}{\partial y_{il}} \bigg)=&\ \varepsilon b\int \nabla W_{\varepsilon,Y}\nabla\frac{\partial {W_{\varepsilon,Y}}}{\partial y_{il}}\bigg(\int |\nabla  W_{\varepsilon,Y}|^2-\int \sum_{t=1}^k|\nabla w_{\varepsilon,y_t}|^2\bigg)\\
  &\ -\bigg(\int Q(x)W_{\varepsilon,Y}^{q-1}\frac{\partial {W_{\varepsilon,Y}}}{\partial y_{il}}-\int \sum_{j=1}^k w_{\varepsilon,y_j}^{q-1}\frac{\partial {W_{\varepsilon,Y}}}{\partial y_{il}}\bigg) \\
  =:&\ \tilde{l}_1-\tilde{l}_2.
\end{aligned}
\end{equation}
By similar estimates of\ $l_1$  as that of Lemma\ \ref{lm3.1}, we have
\begin{eqnarray}\label{521}
  |\tilde{l}_1| \le  C\varepsilon^{2}\sum_{i\not=j}e^{-\frac{\eta|y_i-y_j|}{\varepsilon}}.
\end{eqnarray}
Next, to estimate\ $\tilde{l}_2$, we have
\begin{equation}\label{522}
\begin{aligned}
  \tilde{l}_2&=\int Q(x)\Big(\sum_{j=1}^k w_{\varepsilon ,y_j}\Big)^{q-1}\frac{\partial {w_{\varepsilon,y_i}}}{\partial y_{il}}-\int \sum_{j=1}^k w_{\varepsilon,y_j}^{q-1}\frac{\partial {w_{\varepsilon,y_i}}}{\partial y_{il}}\\
    &= \int \bigg(  \Big(\sum_{j=1}^k w_{\varepsilon ,y_j}\Big)^{q-1}-\sum_{j=1}^k w_{\varepsilon,y_j}^{q-1} \bigg)\frac{\partial {w_{\varepsilon,y_i}}}{\partial y_{il}}+\int\left(Q(x)-1\right) \Big(\sum_{j=1}^k w_{\varepsilon ,y_j}\Big)^{q-1}\frac{\partial {w_{\varepsilon,y_i}}}{\partial y_{il}}\\
  &=: \tilde{l}_{21}+\tilde{l}_{22}.
\end{aligned}
\end{equation}
Since
\begin{eqnarray}\label{523}
  \tilde{l}_{21} \le (q-1)\sum_{j=1,j\not =i}^k\int w_{\varepsilon,y_i}^{q-2} w_{\varepsilon,y_j}\frac {\partial w_{\varepsilon,y_{i}}} {\partial y_{il}}+O\Big(\varepsilon^2\sum_{i\not =j}e^{-\bar\theta\frac{\eta|y_i-y_j|}{\varepsilon}}\Big),
\end{eqnarray}
and
\begin{equation}\label{524}
\begin{aligned}
  \tilde{l}_{22} =&\ \sum_{j=1}^k \int  \left(Q(x)-1\right) w_{\varepsilon,y_j}^{q-1}\frac{\partial {w_{\varepsilon,y_i}}}{\partial y_{il}}+O\Big(\varepsilon^2\sum_{i\not =j}e^{-\bar\theta\frac{\eta|y_i-y_j|}{\varepsilon}}\Big) \\
   =&\ \int (Q(x)-1)w_{\varepsilon,y_i}^{q-1}\frac{\partial {w_{\varepsilon,y_i}}}{\partial y_{il}}+O\Big(\varepsilon^2\sum_{i\not =j}e^{-\bar\theta\frac{\eta|y_i-y_j|}{\varepsilon}}\Big) \\
   =&\ \frac {1}{q}\int \frac{\partial Q(x)}{\partial x_l}w_{\varepsilon,y_i}^q+O\Big(\varepsilon^2\sum_{i\not =j}e^{-\bar\theta\frac{\eta|y_i-y_j|}{\varepsilon}}\Big) \\
   =&\ C\varepsilon^3D_lQ(y_i) +O\Big(\sum_{i=1}^k \sum_{m=2}^{[h]}\varepsilon^{2+m}|D^mQ(y_i)|+\varepsilon^{3+[h]}\Big) \\
   &\ +O\Big(\varepsilon^2\sum_{i\not =j}e^{-\bar\theta\frac{\eta|y_i-y_j|}{\varepsilon}}\Big),
\end{aligned}
\end{equation}
combining\ (\ref{521})--(\ref{524}) yields
\begin{equation}\label{525}
\begin{aligned}
  {l_{\varepsilon ,Y}}\bigg( \frac{\partial {W_{\varepsilon,Y}}}{\partial y_{il}} \bigg)=&\ -C\varepsilon^3D_lQ(y_i) -(q-1)\sum_{j=1,j\not =i}^k\int w_{\varepsilon,y_i}^{q-2} w_{\varepsilon,y_j}\frac {\partial w_{\varepsilon,y_{i}}} {\partial y_{il}}+O(\varepsilon^{3+[h]}) \\
   &\ +O\Big(\sum_{i=1}^k \sum_{m=2}^{[h]}\varepsilon^{2+m}|D^mQ(y_i)|\Big)+O\Big(\varepsilon^2\sum_{i\not =j}e^{-\bar\theta\frac{\eta|y_i-y_j|}{\varepsilon}}\Big).
\end{aligned}
\end{equation}
Next, we estimate the second term of the right side of (\ref{519}). We have
\begin{equation}\label{526}
\begin{aligned}
  \bigg\langle {{I_\varepsilon }''\left( {{W_{\varepsilon ,Y}}} \right)\left[ \varphi_{\varepsilon,Y} \right],\frac{\partial {W_{\varepsilon,Y}}}{\partial y_{il}} } \bigg\rangle =&\ \bigg\langle \varphi_{\varepsilon,Y},\frac{\partial {W_{\varepsilon,Y}}}{\partial y_{il}}  \bigg\rangle_\varepsilon+\varepsilon b\int|\nabla W_{\varepsilon,Y}|^2\int {\nabla \varphi_{\varepsilon,Y}\nabla {\frac{\partial {W_{\varepsilon,Y}}}{\partial y_{il}}}}\\
   &\ +2\varepsilon b \int {\nabla W_{\varepsilon,Y} \nabla \varphi_{\varepsilon,Y}} \int {\nabla W_{\varepsilon,Y} \nabla\frac{\partial {W_{\varepsilon,Y}}}{\partial y_{il}}}\\
   &\ -(q-1)\int Q(x)W_{\varepsilon,Y}^{q-2}\frac{\partial {W_{\varepsilon,Y}}}{\partial y_{il}}\varphi_{\varepsilon,Y}.
\end{aligned}
\end{equation}
Since\ $\varphi_{\varepsilon,Y}\in E^k_{\varepsilon,Y}$,
\begin{eqnarray}\label{530}
\bigg\langle \varphi_{\varepsilon,Y},\frac{\partial {W_{\varepsilon,Y}}}{\partial y_{il}}  \bigg\rangle_\varepsilon =0.
\end{eqnarray}
By\ H\"older inequality, we obtain
\begin{equation}\label{531}
\begin{aligned}
 \varepsilon b\int|\nabla W_{\varepsilon,Y}|^2\int {\nabla \varphi_{\varepsilon,Y}\nabla {\frac{\partial {W_{\varepsilon,Y}}}{\partial y_{il}}}} &\le\varepsilon bk\int\sum_{i=1}^k|\nabla w_{\varepsilon,y_i}|^2\|\nabla \varphi_{\varepsilon,Y}\|_{L^2}\Big\|\nabla {\frac{\partial {W_{\varepsilon,Y}}}{\partial y_{il}}}\Big\|_{L^2}  \\
   &\le C\varepsilon^{\frac 12}\|\varphi_{\varepsilon,Y}\|_{\varepsilon},
 \end{aligned}
 \end{equation}
\begin{equation}\label{532}
\begin{aligned}
  \varepsilon b \int {\nabla W_{\varepsilon,Y} \nabla \varphi_{\varepsilon,Y}} \int {\nabla W_{\varepsilon,Y} \nabla\frac{\partial {W_{\varepsilon,Y}}}{\partial y_{il}}} &\le \varepsilon b\|\nabla W_{\varepsilon,Y}\|^2_{L^2}\|\nabla \varphi_{\varepsilon,Y}\|_{L^2}\Big\|\nabla {\frac{\partial {W_{\varepsilon,Y}}}{\partial y_{il}}}\Big\|_{L^2}   \\
   &\le C\varepsilon^{\frac 12}\|\varphi_{\varepsilon,Y}\|_{\varepsilon},
 \end{aligned}
 \end{equation}
 \begin{equation}\label{533}
 \begin{aligned}
   \int Q(x)W_{\varepsilon,Y}^{q-2}\frac{\partial {W_{\varepsilon,Y}}}{\partial y_{il}}\varphi_{\varepsilon,Y}&\le C\bigg(\int W_{\varepsilon,Y}^{q}\bigg)^{\frac{q-2}{q}}\bigg(\int|\frac{\partial {W_{\varepsilon,Y}}}{\partial y_{il}}|^{q}\bigg)^{\frac{1}{q}}\bigg(\int|\varphi_{\varepsilon,Y}|^{q}\bigg)^{\frac{1}{q}}   \\
   &\le C\varepsilon^{\frac 12}\|\varphi_{\varepsilon,Y}\|_{\varepsilon}.
 \end{aligned}
 \end{equation}
Combining\ (\ref{526})--(\ref{533}) yields
\begin{eqnarray}\label{534}
  \bigg\langle {{I_\varepsilon }''\left( {{W_{\varepsilon ,Y}}} \right)\left[ \varphi_{\varepsilon,Y} \right],\frac{\partial {W_{\varepsilon,Y}}}{\partial y_{il}} } \bigg\rangle = O(\varepsilon^{\frac 12}\|\varphi_{\varepsilon,Y}\|_{\varepsilon}).
\end{eqnarray}
Besides, by Lemma\ \ref{lm3.2}, we have
\begin{eqnarray}\label{527}
\bigg\langle {{R'_{\varepsilon ,Y}}( \varphi_{\varepsilon,Y}  )},{\frac{\partial {W_{\varepsilon,Y}}}{\partial y_{il}}} \bigg\rangle =o_\varepsilon(1)\|\varphi_{\varepsilon,Y}\|_{\varepsilon}.
\end{eqnarray}
By Lemma\ \ref{lm3.5}, we have
\begin{eqnarray}\label{528}
\bigg\langle {\frac {\partial J_\varepsilon} {\partial \varphi_{\varepsilon,Y}}},{\frac {\partial \varphi_{\varepsilon,Y}} {\partial y_{il}}} \bigg\rangle =0.
\end{eqnarray}
Combining\ (\ref{525}) and\ (\ref{534})--(\ref{528}) yields\ (\ref{517}).
\end{proof}

\medskip

 \noindent\textbf{Acknowledgments}
The authors would like to thank Chunhua Wang from Central China Normal University for the helpful discussion with her. This paper was supported by NSFC grants (No.12071169).


\begin{thebibliography}{32}
 {\footnotesize

\bibitem{20} G. Kirchhoff, Mechanik [M]. Teubner, Leipzig, 1883.

\bibitem{4} S. Bernstein, Sur une classe d'\'equations fonctionelles aux d\'eriv\'ees partielles [J]. {\it Bull. Acad. Sci. URSS. S\'er.} {\bf 4}(1940), 17--26.

\bibitem{31} S. I. Pohozaev, A certain class of quasilinear hyperbolic equations [J]. {\it Mat. Sb. (N.S.)} {\bf96}(138)(1975), 152--166, 168 (in Russian).

\bibitem{26} J. L. Lions, On some questions in boundary value problems of mathematical physics. Contemporary Developments in Continuum Mechanics and Partial
    Differential Equations [M]. North-Holland Math. Stud., Vol. 30, North-Holland, Amsterdam--New York, 1978, pp. 284--346.

\bibitem{2} A. Arosio, S. Panizzi, On the well-posedness of the Kirchhoff string [J]. {\it Trans. Amer. Math. Soc.} {\bf348}(1996), 305--330.

\bibitem{11} X. M. He, W. M. Zou, Existence and concentration behavior of positive solutions for a Kirchhoff equation in $\mathbb R^3$ [J]. {\it J. Differ. Equ.} {\bf 252}(2012), 1813--1834.

\bibitem{12} G. M. Figueiredo, N. Ikoma, J. R. Santos J\'unior, Existence and concentration result for the Kirchhoff type equations with general nonlinearities [J]. {\it Arch. Rational Mech. Anal.} {\bf213}(2014), 931--979.

\bibitem{15} Z. Guo, Ground states for Kirchhoff equations without compact condition [J]. {\it J. Differ. Equ.} {\bf259}(2015), 2884--2902.

\bibitem{16} Y. Deng, S. Peng, W. Shuai, Existence and asymptotic behavior of nodal solutions for the Kirchhoff-type problems in $\mathbb R^3$ [J]. {\it J. Funct. Anal.} {\bf 269}(2015), 3500--3527.

\bibitem{43} S. Xu, C. Wang, Local uniqueness of a single peak solution of a subcritical Kirchhoff problem in $\mathbb R^3$ [J]. {\it (Chinese. English) Acta Math. Sci. Ser. A (Chinese Ed.)} {\bf 40} (2020), 432--440.

\bibitem{30} K. Perera, Z. Zhang, Nontrivial solutions of Kirchhoff-type problems via the Yang index [J]. {\it J. Differ. Equ.} {\bf221}(2006), 246--255.

\bibitem{19} Y. He, G. Li, S. Peng, Concentrating bound states for Kirchhoff type problems in $\mathbb R^3$involving critical Sobolev exponents [J]. {\it Adv. Nonlinear Stud.} {\bf 14}(2014), 483--510.

\bibitem{18} Y. He, G. Li, Standing waves for a class of Kirchhoff type problems in $\mathbb R^3$ involving critical Sobolev exponents [J]. {\it Calc. Var. Partial Differ. Equ.} {\bf 54}(2015), 3067--3106.

\bibitem{17} Y. He, Concentrating bounded states for a class of singularly perturbed Kirchhoff type equations with a general nonlinearity [J]. {\it J. Differ. Equ.}  {\bf 261}(2016), 6178--6220.

\bibitem{21} M. K. Kwong, Uniqueness of positive solutions of $\Delta u-u+u^p=0$ in $\mathbb R^n$ [J]. {\it Arch. Rational Mech. Anal.} {\bf 105}(1989), 243--266.

\bibitem{6} D. Cao, E. S. Noussair, S.Yan,  Solutions with multiple peaks for nonlinear elliptic equations [J]. {\it Proc. Roy. Soc. Edinburgh Sect. A} {\bf129}(1999), 235--264.

\bibitem{24} E.N. Dancer, S. Yan, On the existence of multipeak solutions for nonlinear field equations on $\mathbb R^ N$ [J]. {\it Discrete Contin. Dynam. Systems} {\bf6}(2000), 39--50.

\bibitem{38} J. Wang, L. Tian, J. Xu and F. Zhang. Multiplicity and concentration of positive solutions
for a Kirchhoff type problem with critical growth [J]. {\it J. Differ. Equ. } {\bf 253} (2012), 2314--2351.

\bibitem{37} G. M. Figueiredo and J. R. Santos J\'unior. Multiplicity and concentration behavior of positive solutions for a Schr\"odinger-Kirchho type problem via penalization method [J]. {\it  ESAIM Control Optim. Calc. Var. } {\bf 20} (2014), 389--415.

\bibitem{7} D. Cao, S. Li, P. Luo, Uniqueness of positive bound states with multi-bump for nonlinear Schr\"odinger equations [J]. {\it Calc. Var. Partial Differ. Equ.} {\bf54}(2015), 4037--4063.

\bibitem{14} C. Gui, Existence of multi-bump solutions for nonlinear Schr\"odinger equations via variational method [J]. {\it Commun. Part. Differ. Equ.} {\bf21}(1996), 787--820.

\bibitem{10} M. del Pino, P. L. Felmer, Multi-peak bound states for nonlinear Schr\"odinger equations [J]. {\it Ann. Inst. H. Poincar\'e C Anal. Non Lin\'eaire} {\bf15}(1998), 127--149.

\bibitem{13} A. Floer, A. Weinstein, Nonspeading wave packets for the cubic Schr\"odinger equation with a bounded potential [J]. {\it J. Funct. Anal.} {\bf69}(1986), 397--408.

\bibitem{35} W. Liu, C. Wang, Multi-peak solutions of a nonlinear Schr\"odinger equation with magnetic fields [J]. {\it Adv. Nonlinear Stud.} {\bf 14}(2014), 951--975.

\bibitem{27} E. S. Noussair, S. Yan, On positive multipeak solutions of a nonlinear elliptic problem [J]. {\it J. London Math. Soc.} (2){\bf62}(2000), 213--227.

\bibitem{29} Y. G. Oh, On positive multi-lump bound states of nonlinear Schr\"odinger equations under multiple well potential [J]. {\it Commun. Math. Phys.} {\bf 131}(1990), 223--253.

\bibitem{23} G. Li, P. Luo, S. Peng, C. Wang, C. Xiang, Uniqueness and Nondegeneracy of positive solutions to Kirchhoff equations and its applications in singular perturbation problems [J]. {\it J. Differ. Equ.} {\bf268}(2020), 541--589.

\bibitem{25} P. Luo, S. Peng, C. Wang, C. Xiang, Multi-peak positive solutions to a class of Kirchhoff equations [J]. {\it  Proc. Roy. Soc. Edinburgh Sect. A}  {\bf149}(2019),1097--1122.

\bibitem{32} L. Cui, G. Li, P. Luo and C. Wang, Existence and local uniqueness of normalized multi-peak solutions to a class of Kirchhoff type equations [J], {\it Minimax Theory Appl.} {\bf7}(2022), 207--252.


\bibitem{1} A. Ambrosetti, E. Colorado, D. Ruiz, Multi-bump solitons to linearly coupled systems of nonlinear Schr\"odinger equations [J]. {\it Calc. Var. Partial Differential Equations.} {\bf30}(2007), 85--112.

\bibitem{22} E.N. Dancer, K.Y. Lam, S. Yan, The effect of the graph topology on the existence of multipeak solutions for nonlinear Schr\"odinger equations [C]. {\it Abstr. Appl. Anal.} {\bf3}(1998), 293--318.

\bibitem{8} E.N. Dancer, S. Yan, Interior and boundary peak solutions for a mixed boundary value problem [J]. {\it Indiana Univ. Math. J.} {\bf48}(1999), 1177--1212.

\bibitem{3} T. Bartsch, S. Peng, Semiclassical symmetric Schr\"odinger equations: existence of solutions concentrating simultaneously on several spheres [J]. {\it Z. Angew. Math. Phys.} {\bf58}(2007), 778--804.
}
\end{thebibliography}
\end{document}